\documentclass[final]{siamltex}
\usepackage{epsfig}
\usepackage{amsmath}
\usepackage{amsfonts}
\usepackage{algorithm}
\usepackage{algorithmic}
\usepackage{graphicx}
\usepackage{subfigure}

\newcommand {\C}        {{\mathbb{C}}}
\newcommand {\R}        {{\mathbb{R}}}
\newcommand {\N}        {{\mathbb{N}}}

\title{ Numerical Optimization of Eigenvalues of Hermitian
	 Matrix Functions }
\author{
 Emre~Mengi\thanks{ Department of Mathematics, Ko\c{c} University,
Rumelifeneri Yolu, 34450 Sar{\i}yer, \.{I}stanbul, Turkey {\tt
(emengi@ku.edu.tr)}. 
The work of this author was supported in part by the European Commission grant
PIRG-GA-268355 and the T\"{U}B\.{I}TAK (The Scientific and Technological Research Council
of Turkey) Career Grant 109T660.
 }
\and
 E.~Alper~Yildirim\thanks{ Department of Industrial Engineering,
Ko\c{c} University,
Rumelifeneri Yolu, 34450 Sar{\i}yer-\.{I}stanbul, Turkey {\tt
(alperyildirim@ku.edu.tr)}. This author was supported in part by T{\"U}B{\.I}TAK
(The Scientific and Technological Research Council of Turkey) 
  Grant 112M870 and by T{\"U}BA-GEB{\.I}P (Turkish Academy of Sciences Young
 Scientists Award Program).}
 \and
   Mustafa~Kili\c{c}\thanks{ Department of Mathematics, Ko\c{c} University,
 Rumelifeneri Yolu, 34450 Sar{\i}yer, \.{I}stanbul, Turkey {\tt
(mukilic@ku.edu.tr)}.
The work of this author was partly supported by the European Commission Grant
PIRG-GA-268355. }
}

\begin{document}
\maketitle

\begin{abstract}
\noindent
This work concerns the global minimization of a prescribed eigenvalue or a weighted sum of
prescribed eigenvalues of a Hermitian matrix-valued function depending on
its parameters analytically in a box. We describe how the analytical properties 
of eigenvalue functions can be put into use to derive piece-wise quadratic functions 
that underestimate the eigenvalue functions. These piece-wise quadratic under-estimators
lead us to a global minimization algorithm, originally due to Breiman and Cutler.
We prove the global convergence of the algorithm, and show that it can be effectively
used for the minimization of extreme eigenvalues, e.g., the largest eigenvalue or
the sum of the largest specified number of eigenvalues. This is particularly 
facilitated by the analytical formulas for the first derivatives of eigenvalues,
as well as analytical lower bounds on the second derivatives that can be
deduced for extreme eigenvalue functions. The applications that we have in mind
also include the ${\rm H}_\infty$-norm of a linear dynamical system, numerical
radius, distance to uncontrollability and various other non-convex eigenvalue
optimization problems, for which, generically, the eigenvalue function involved is simple 
at all points. \\


\end{abstract}

\pagestyle{myheadings}
\thispagestyle{plain}
\markboth{E. MENGI, E.A. YILDIRIM, AND M. KILI\c{C}}{
Optimization of Hermitian Eigenvalues }

\section{Introduction}
The main object of this work is a matrix-valued function 
${\mathcal A}(\omega) : \R^d \rightarrow \C^{n\times n}$ that is
\emph{analytic} and \emph{Hermitian} at all $\omega\in{\mathbb R}^d$.
Here, we consider the numerical \emph{global} minimization of a prescribed 
eigenvalue $\lambda(\omega)$ of ${\mathcal A}(\omega)$ over $\omega \in {\mathcal B} \subseteq \R^d$,
where ${\mathcal B}$ denotes a box. From an application point of
view, a prescribed eigenvalue typically refers to the $j$th largest eigenvalue, i.e., 
$\lambda(\omega) := \lambda_j({\mathcal A}(\omega))$, or a weighted sum of
$j$ largest eigenvalues, i.e., $\lambda(\omega) := \sum_{k=1}^j d_k \lambda_{k} ({\mathcal A}(\omega))$
for given real numbers $d_1,\dots,d_j$. 
However, it may as well refer to a particular eigenvalue with respect to a different criterion as 
long as the (piece-wise) analyticity properties discussed below and in Section \ref{sec:Her_eig_pert} are satisfied.

The literature from various engineering fields and applied sciences is rich with eigenvalue
optimization problems that fits into the setting of the previous paragraph. There are problems arising in structural 
design and vibroacoustics, for which the minimization of the largest eigenvalue or maximization of the smallest 
eigenvalue of a matrix-valued function is essential, e.g., the problem of designing the strongest column which originated 
from Euler in the 18th century \cite{Lewis1996}. In control theory, various quantities regarding dynamical systems 
can be posed as eigenvalue optimization problems. For instance, the distance from a linear dynamical system to a nearest 
unstable system \cite{VanLoan1984}, and the ${\rm H}_\infty$-norm of a linear dynamical system have non-convex eigenvalue 
optimization characterizations \cite{Boyd1990}. In graph theory, relaxations of some NP-hard graph partitioning problems give rise 
to optimization problems in which the sum of the $j$ largest eigenvalues is to be minimized \cite{Donath1973}. 

In this paper, we offer a generic algorithm based on the analytical properties of eigenvalues of an analytic and Hermitian
matrix-valued function, that is applicable for any eigenvalue optimization problem whenever lower bounds 
on the second derivatives of the eigenvalue function can be calculated analytically or numerically. All of the existing 
global eigenvalue optimization algorithms in the non-convex setting are designed for specific problems, e.g., 
\cite{Boyd1990, Bruinsma1990, Byers1988, Byers1993, Freitag2011, Gao1993,Gu2000,Gu2006,He1999,Hinrichsen1989, Mengi2005}, 
while widely adopted techniques such as interior point 
methods \cite{Nesterov1994} - when it is possible to pose an eigenvalue optimization problem 
as a semi-definite program - or a bundle method \cite{Makela2002} are effective in the 
convex setting. We foresee non-convex eigenvalue optimization problems
that depend on a few parameters as the typical setting for the use of the algorithm here.

For the optimization of non-convex eigenvalue functions, it appears essential to benefit from the global 
properties of eigenvalue functions, such as their global Lipschitzness or global bounds on their 
derivatives. Such global properties lead us to approximate $\lambda(\omega)$ globally with
under-estimating functions, which we call support functions. Furthermore, the derivatives of the
eigenvalue functions can be evaluated effectively at no cost once the eigenvalue function is
evaluated (due to analytic expressions for the derivatives of eigenvalues in terms of 
eigenvectors as discussed in Section \ref{subsec:first_der}). Therefore, the incorporation of the derivatives 
into the support functions yields \emph{quadratic support functions} on which our algorithm relies.
The quadratic support functions for eigenvalue functions are derived exploiting the analytical 
properties of eigenvalues and presume the availability of a lower bound $\gamma$ on the second 
derivatives of the eigenvalue function that is obtained either analytically or numerically.

\vskip 1ex

\noindent
\textit{\underline{Example:}  Consider the minimization of the largest eigenvalue 
$\lambda_1(\omega) = \lambda_1\left( {\mathcal A}(\omega) \right)$ of
\[
	{\mathcal A} : {\mathbb R} \rightarrow {\mathbb R}^{n\times n}, \;\;\;
	{\mathcal A}(\omega) 	:= 		A_0		+	\omega
A_1	+	\omega^2	A_2,
\]
where $A_0, A_1, A_2 \in {\mathbb R}^{n\times n}$ are given symmetric matrices. It can be deduced 
from the expressions in Section \ref{subsec:sec_der} that
$
	\lambda_1''(\omega)	  \geq	\gamma := 2 \lambda_{\min}
\left(   A_2  \right)
$
for all $\omega$ such that $\lambda_1(\omega)$ is simple. Furthermore, due to expressions
in Section \ref{subsec:first_der}, at all such $\omega$, we have
$
	\lambda_1'(\omega) = v_1(\omega)^T (A_1 + 2\omega A_2) \; v_1(\omega),
$
where $v_1(\omega)$ is a unit eigenvector associated with $\lambda_1(\omega)$. Consequently, it turns out
that, about any $\omega_k \in {\mathbb R}$ where $\lambda_1(\omega_k)$ is simple, there is a support function
\[
	q(\omega) := \lambda_1(\omega_k)	 +  \lambda_1'(\omega_k) (\omega - \omega_k)  + \frac{\gamma}{2} (\omega - \omega_k)^2
\]
satisfying $q(\omega) \leq \lambda_1(\omega)$ for all $\omega\in {\mathbb R}$; see Section \ref{subsec:support_extreme}
for the details.
% Furthermore, for generic matrices $A_0, A_1, A_2$ the eigenvalue function $\lambda_1(\omega)$ do not intersect
% $\lambda_2(\omega)$ for all $\omega$ making it analytic everywhere
}

\vskip 1ex

Support functions have earlier been explored by the global optimization
community. The Piyavskii-Shubert 
algorithm \cite{Piyavskii1972, Shubert1972} is derivative-free, and constructs conic support 
functions based on Lipschitz continuity with a known global Lipschitz constant. It converges 
sub-linearly in practice. Sophisticated variants that make use of several Lipschitz constants
simultaneously appeared in the literature \cite{Jones1993,  Sergeyev2006}. The idea of using 
derivatives in the context of global optimization yields powerful algorithms. Breimann
and Cutler \cite{Breiman1993} developed an algorithm that utilizes quadratic support functions 
depending on the derivatives. Some variants of the Breimann-Cutler algorithm are also suggested
for functions with Lipschitz-continuous derivatives;  for instance \cite{Gergel1997, Kvasov2009, Kvasov2012} benefit 
from multiple Lipschitz constants for the derivatives, \cite{Sergeyev1999} estimates Lipschitz constants 
for the derivatives locally, while \cite{Lera2013} modifies the support functions of the Breimann-Cutler
algorithm in the univariate case so that the subproblems become smooth;  however, all these variants 
in the multivariate case end up working on a mesh as a downside. The quadratic support functions that 
we derive for $\lambda(\omega)$ coincide with the quadratic support functions on which the 
Breimann-Cutler algorithm is built on. Consequently, our approach is a variant of the algorithm 
due to Breimann and Cutler \cite{Breiman1993}. 

At every iteration of the algorithm, a global minimizer of a piece-wise
quadratic model defined as the maximum of a set of quadratic support functions is determined.
A new quadratic support function is constructed around this global minimizer, and the
piece-wise quadratic model is refined with the addition of this new support function.
In practice, we observe a linear rate of convergence to a global minimizer.

The algorithm appears applicable especially to extremal eigenvalue functions of the form
\[
		\lambda(\omega) = \sum_{k=1}^j d_k \lambda_k \left( {\mathcal
A}(\omega) \right),
\]
where $d_k$ are given real numbers such that $d_1 \geq d_2 \geq \dots \geq d_j \geq 0$.
This is facilitated by the simple quadratic support functions derived in Section \ref{subsec:support_extreme},
and expressions for the lower bound $\gamma$ on the second derivatives derived
in Section \ref{sec:extreme_low_bound}. The algorithm is also applicable if the eigenvalue
function $\lambda(\omega)$ is simple over all $\omega \in {\mathcal B}$,
which holds for various eigenvalue optimization problems of interest.

\vskip 1ex

\noindent
\textbf{Outline:}
We start in the next section with a list of eigenvalue optimization problems to which
our proposed algorithm fits well. In Section \ref{sec:Her_eig_pert}, the basic results concerning 
the analyticity and derivatives of the eigenvalues of a Hermitian matrix-valued function ${\mathcal A}(\omega)$
that depends analytically on $\omega$ are reviewed. 
In Section \ref{sec:quad_support}, for a general eigenvalue function,
the piece-wise quadratic support functions that are defined as the minimum of $n$ quadratic functions 
are derived. In Section \ref{sec:squad_support}, it is shown that these piece-wise quadratic support functions 
simplify to smooth quadratic support functions for the extremal eigenvalue
functions, as well as for the eigenvalue functions that are simple for all $\omega \in {\mathcal B}$.
Global lower bounds $\gamma$ on the second derivatives of an extremal eigenvalue
function are deduced in Section \ref{sec:extreme_low_bound}.
The algorithm based on the quadratic support functions is presented in Section \ref{sec:algorithm}.
We establish the global convergence of the proposed algorithm in Section
\ref{sec:1d_anal}. Finally, comprehensive
numerical experiments are provided in Section \ref{sec:eig_opt}. The examples indicate the
superiority of the algorithm over the Lipschitz continuity based algorithms, e.g., \cite{Jones1993, Piyavskii1972, Shubert1972}, 
as well as the level-set based approaches devised for particular non-convex eigenvalue optimization
problems, e.g., \cite{Gu2006, Mengi2005}. The reader who prefers to avoid technicalities at first
could glance at the algorithm in Section \ref{sec:algorithm}, then go through 
Sections \ref{sec:Her_eig_pert}-\ref{sec:extreme_low_bound} for the theoretical foundation.

 \section{Applications}\label{sec:app}

\subsection{Quantities Related to Dynamical Systems}\label{sec:app_dyn_sys}
The \emph{numerical radius} $r(A)$ of $A\in {\mathbb C}^{n\times n}$ is the modulus
of the outer-most point in its field of values \cite{Horn1991}, and is defined by
\[
	r(A) := \max
		\{
			| z^{\ast} A z |  \; | \;  z\in {\mathbb C}^n \;\;  {\rm s.t.}  \;\;	\| z \|_2 = 1
		\}.
\]
This quantity gives information about the powers of $A$, e.g.,$\| A^k \| \leq 2r(A)^k$, and is used 
in the literature to analyze the convergence of iterative methods for the solution of linear systems 
\cite{Axelsson1994, Eiermann1993}. An eigenvalue optimization characterization
is given by \cite{Horn1991}:
\[
	r(A)
		=
	- 
	\left[
	\min_{\theta \in [0,2\pi]} \;
	\lambda_n
		\left(
			{\mathcal A}(\theta)
		\right)
	\right], \;\;\;\;
	{\mathcal A}(\theta) := -(Ae^{i\theta} + A^{\ast}e^{-i\theta})/2.
\]

The ${\rm H}_{\infty}$-norm is one of the most widely used norms in practice for the descriptor system
\[
	Ex'(t) =  A x(t) + Bu(t), \;\; {\rm and} \;\;
	y(t) =  C x(t) + Du(t), 
\]
where $u(t)$ and $y(t)$ are the input and output functions, respectively, and 
$E, A \in {\mathbb C}^{n\times n}$, $B \in{\mathbb C}^{n\times m}, C\in {\mathbb C}^{p\times n}, D\in {\mathbb C}^{p\times m}$ with $m, p \leq n$
are the system matrices.  The ${\rm H}_{\infty}$-norm of the transfer function for this system is defined as
\[
	\| H \|_\infty 	:=	\frac{1}{\inf_{\omega\in {\mathbb R}} \;	\sigma_n \left[ H(i\omega)^{\dagger} \right]},
	\;\;\;\;
	H(s) := \left[	C (sE - A)^{-1} B +D \right].
\]
Here and elsewhere, $\sigma_j(\cdot)$ represents the $j$th largest singular value, and
$H(i\omega)^{\dagger}$  denotes the pseudoinverse of $H(i\omega)$. Also above,
with zero initial conditions for the descriptor system,  the transfer function $H(s)$ reveals the 
linear relation between the input and output, as
$
	Y(s)	= H(s) U(s),
$
with $U(s)$ and $Y(s)$ denoting the Laplace transformations of $u(t)$ and 
$y(t)$, respectively.
Note that the ${\rm H}_{\infty}$-norm above is ill-posed (i.e., the associated operator is unbounded) if the 
pencil $L(\lambda) = A - \lambda E$ has an eigenvalue on the imaginary axis or to the
right of the imaginary axis. Therefore, when the ${\rm H}_\infty$-norm is well-posed, the matrix-valued
function ${\mathcal A}(\omega) := H(i\omega)$ is analytic at all $\omega \in
{\mathbb R}$.
A relevant quantity is the (continuous) distance to instability from a matrix $A \in{\mathbb C}^{n\times n}$; 
the eigenvalue optimization characterization for the ${\rm H}_\infty$-norm with $E = B = C = I_n$ and $D = 0$ 
reduces to that for the distance to instability \cite{VanLoan1984} from $A$ with respect to the
$\ell_2$-norm.

Paige \cite{Paige1981} suggested the distance to uncontrollability, for a given $A\in {\mathbb C}^{n\times n}$
and $B\in {\mathbb C}^{n\times m}$ with $m \leq n$, defined by
\[
	\tau(A,B)
		:=
	\inf
		\left\{
			\left\|
			\left[
				\begin{array}{cc}
					\Delta A 	&	\Delta B
				\end{array}
			\right] \right\|_2 \; | \;
			(A+\Delta A, B+\Delta B) \;\; {\rm is} \;\;
			{\rm uncontrollable}
		\right\},
\]
as a robust measure of controllability. Here, the controllability of a linear control system $(A,B)$ of the form $x'(t) = Ax(t) + Bu(t)$ 
means that the function $x(t)$ can be driven into any state at a particular time by some input $u(t)$, and could be 
equivalently characterized as
$
	{\rm rank}
		\left(
			\left[
				\begin{array}{cc}
					A - z I	&	B
				\end{array}
			\right]
		\right) = n,	\;\;\;	\forall z \in {\mathbb C}.
$
Therefore, the eigenvalue optimization characterization for the distance
to uncontrollability takes the form \cite{Eising1984}:
\[
	\tau(A,B)	=
	\min_{z \in {\mathbb C}}
		\sigma_n
			\left(
				{\mathcal A}(z)
			\right), \;\;\;\;
	{\mathcal A}(z) := 
			\left[
				\begin{array}{cc}
					A - z I	&	B
				\end{array}
			\right].
\]

\subsection{Minimizing the Largest or Maximizing the Smallest Eigenvalues}\label{sec:app_min_largest}
In the 18th century, Euler considered the design of the strongest column with a
given
volume with respect to the radii of the cross-sections \cite{Lewis1996, Overton1992}.
The problem can be formulated as finding the parameters, representing the
radii of cross-sections, maximizing the smallest eigenvalue of a fourth order differential 
operator. The analytical solution of the problem has been considered in several
studies in 1970s and in 1980s \cite{Barnes1988,Myers1986,Olhoff1977}, which
were motivated by the earlier work of Keller and Tadjbakhsh
\cite{Tadjbakhsh1962}.
Later, the problem is treated numerically \cite{Cox1992} by means of the
finite-element 
discretization, giving rise to the problem
\begin{equation}\label{eq:minimize_largest}
	{\rm min}_{\omega\in {\mathbb R}^d} \;	\lambda_1 \left(  {\mathcal A}(\omega)   \right).
\end{equation}
The treatment in \cite{Cox1992} yields ${\mathcal A}(\omega)	:=	A_0	+	\sum_{j=1}^d \omega_j A_j$.
In this affine setting, the minimization of the largest eigenvalue is a convex
optimization problem (immediate from Theorem \ref{thm:extreme_low_bound} below)
and received considerable attention \cite{Fletcher1985, Friedland1987, Overton1988}.

In the general setting, when the dependence of the matrix function ${\mathcal A}(\omega)$
on the parameters is not affine, the problem in (\ref{eq:minimize_largest}) is non-convex.
Such non-convex problems are significant (though they are not studied much excluding
a few studies such as \cite{Overton1995} that offer only local analysis) in robust
control theory for 
instance to ensure robust stability. The dual form that concerns the maximization of the smallest 
eigenvalue is of interest in vibroacoustics.

\subsection{Minimizing the Sum of the $j$ Largest
Eigenvalues}\label{sec:app_min_sum_largest}
In graph theory, relaxations of the NP-hard partitioning problems lead to eigenvalue 
optimization problems that require the minimization of the sum of the $j$
largest eigenvalues. For instance,  
%One such problem is that, 
given a weighted graph with $n$ vertices and nonnegative integers
$d_1 \geq d_2 \geq \dots \geq d_j$ summing up to $n$,  consider finding a partitioning of the graph 
such that the $\ell$th partition contains exactly $d_\ell$ vertices for $\ell = 1,\dots, j$ and the sum 
of the weights of the edges within each partition is maximized. The relaxation
of this
problem suggested in \cite{Donath1973} is of the form
\begin{equation}\label{eq:minimize_jlargest}
	{\rm min}_{\omega\in {\mathbb R}^d} \;	\sum_{k=1}^j  d_k
				\lambda_k
					\left(
						{\mathcal A}(\omega)
					\right).
\end{equation}
The problem (\ref{eq:minimize_jlargest}) is convex, if ${\mathcal A}(\omega)$ is an affine function 
of $\omega$, as in the case considered by \cite{Donath1973}, see also \cite{Cullum1975}.

Once again, in general, the minimization of the sum of the $j$ largest
eigenvalues is not a 
convex optimization problem, and there are a few studies in the literature that
attempted to analyze the problem
locally for instance around the points where the eigenvalues coalesce \cite{Shapiro1995}.

\section{Background on Perturbation Theory of Eigenvalues}\label{sec:Her_eig_pert}
In this section, we first briefly summarize the analyticity results, mostly
borrowed from \cite[Chapter 1]{Rellich1969}, related to the eigenvalues of matrix-valued functions. 
Then, expressions \cite{Lancaster1964} are provided for the derivatives of Hermitian eigenvalues  
in terms of eigenvectors and the derivatives of matrix-valued functions. Finally, 
we elaborate on the analyticity of singular value problems as special Hermitian eigenvalue problems.

\subsection{Analyticity of Eigenvalues}
\subsubsection{Univariate Matrix Functions}
For a univariate matrix-valued function ${\mathcal A}(\omega)$ that depends 
on $\omega$ analytically, which may or may not be Hermitian, the characteristic
polynomial is of the form
\[
	g(\omega,\lambda) := \det ( \lambda I - {\mathcal A}(\omega) )  =
	a_n(\omega) \lambda^n + \dots + a_1(\omega) \lambda + a_0(\omega),
\]
where $a_0(\omega), \dots, a_n(\omega)$ are analytic functions of $\omega$. It follows
from the Puiseux' theorem (see, e.g., \cite[Chapter 2]{Wall2004})
that each root
$\tilde{\lambda}_j(\omega)$ such that $g(\omega, \tilde{\lambda}_j(\omega)) = 0$
has a Puiseux series of the form
\begin{equation}\label{eq:Puiseux}
	\tilde{\lambda}_j(\omega) = \sum_{k = 0}^{\infty} c_{k,j} \omega^{k/r},
\end{equation}
for all small $\omega$, where $r$ is the multiplicity of the root
$\tilde{\lambda}_j(0)$.

Now suppose ${\mathcal A}(\omega)$ is Hermitian for all $\omega$, and let $\ell$
be the smallest integer such that $c_{\ell,j} \neq 0$. Then, we have
\[
	\lim_{\omega \rightarrow 0^+} \frac{\tilde{\lambda}_j(\omega) -
\tilde{\lambda}_j(0)}{\omega^{\ell/r}} = c_{\ell,j},
\]
which implies that $c_{\ell,j}$ is real, since $\tilde{\lambda}_j(\omega)$ and 
$\omega^{\ell/r}$ are real numbers for each $\omega$.  Furthermore,
\[
	\lim_{\omega \rightarrow 0^-} \frac{\tilde{\lambda}_j(\omega) - \tilde{\lambda}_j(0)}{(-\omega)^{\ell/r} }
							=
		(-1)^{\ell/r} c_{\ell,j}			
\]
is real, which implies that $(-1)^{\ell/r}$ is real, or equivalently
that $\ell/r$ is integer. This observation reveals that the first nonzero
term in the Puiseux series of $\tilde{\lambda}_j(\omega)$ is an integer power of $\omega$.
The same argument applied to the derivatives of $\tilde{\lambda}_j(\omega)$ and the
associated Puiseux series indicates that only integer powers of $\omega$ can appear in 
the Puiseux series (\ref{eq:Puiseux}), that is the Puiseux series reduces to a power series. 
This establishes that $\tilde{\lambda}_j(\omega)$ is an analytic function of $\omega$. 
Indeed, it can also be deduced that, associated with
$\tilde{\lambda}_1(\omega), \dots, \tilde{\lambda}_n(\omega)$, there is an orthonormal set
$\{v_1(\omega), \dots, v_n(\omega) \}$ of eigenvectors, where each of $v_1(\omega), \dots, v_n(\omega)$
varies analytically with respect to $\omega$ (see \cite{Rellich1969} for details).

\begin{theorem}[Rellich]\label{thm:Rellich}
	Let ${\mathcal A}(\omega) : {\mathbb R} \rightarrow {\mathbb C}^{n\times n}$ be a
	Hermitian matrix-valued function that depends on $\omega$ analytically. 
	\begin{itemize}
		\item[\bf (i)] The $n$ roots of the characteristic polynomial of ${\mathcal A}(\omega)$ can 
		be arranged so that each root $\tilde{\lambda}_j(\omega)$ for $j=1,\dots,n$ is an analytic 
		function of $\omega$.
		\item[\bf (ii)] There exists an eigenvector $v_{j}(\omega)$ associated
		with $\tilde{\lambda}_j(\omega)$ for $j = 1,\dots,n$
that satisfies the following:
			\begin{enumerate}
				\item[\bf (1)] 
					$\left(
						\tilde{\lambda}_j(\omega) I - {\mathcal A}(\omega)
					  \right) v_j(\omega)	= 0, \;\;
\forall \omega\in {\mathbb R}$,
				\item[\bf (2)] $\| v_j(\omega) \|_2 = 1, \;\;
\forall \omega\in {\mathbb R}$, 
				\item[\bf (3)] $v_j^\ast (\omega) v_k(\omega) =
0, \;\; \forall \omega \in {\mathbb R}$ for $k \neq j$, and 
				\item[\bf (4)] $v_j(\omega)$ is an analytic function of $\omega$.
			\end{enumerate} 
	\end{itemize}
\end{theorem}

\subsubsection{Multivariate Matrix Functions} \label{subsec:anal_mult_mat_func}
The eigenvalues of a multivariate matrix-valued function 
${\mathcal A}(\omega) : {\mathbb R}^d \rightarrow {\mathbb C}^{n\times n}$ that depends on 
$\omega$ analytically do not have a power series representation in general even when 
${\mathcal A}(\omega)$ is Hermitian. As an example, consider
\[
	{\mathcal A}(\omega)
				=
			\left[
				\begin{array}{cc}
					\omega_1  					&   \frac{\omega_1 + \omega_2}{2}   \\
					\frac{\omega_1 + \omega_2}{2}	&   \omega_2 \\
				\end{array}
			\right]
			\;\;\;	{\rm with}		\;\;\;
	\tilde{\lambda}_{1,2}(\omega)  = \frac{\omega_1 + \omega_2}{2} \pm \sqrt{\frac{\omega_1^2 + \omega_2^2}{2}}.	
\]
On the other hand, it follows from Theorem \ref{thm:Rellich} that, there are underlying eigenvalue functions 
$\tilde{\lambda}_j(\omega), \; j=1,\dots,n$, of ${\mathcal A}(\omega)$, each of
which is analytic 
along every line in ${\mathbb R}^d$, when ${\mathcal A}(\omega)$ is Hermitian. This analyticity 
property along lines in ${\mathbb R}^d$ implies the existence of the first partial derivatives of 
$\tilde{\lambda}_j(\omega)$ everywhere. Expressions for the first partial derivatives will be derived in the 
next subsection, indicating their continuity. As a consequence of the continuity of the first partial derivatives, 
each $\tilde{\lambda}_j(\omega)$ must be differentiable.
 
\begin{theorem}\label{thm:anal_mul_var}
	Let ${\mathcal A}(\omega) : {\mathbb R}^d \rightarrow {\mathbb C}^{n\times n}$ be a
	Hermitian matrix-valued function that depends on $\omega$ analytically.
Then, the
	$n$ roots of the characteristic polynomial of ${\mathcal A}(\omega)$ can be arranged
	so that each root $\tilde{\lambda}_j(\omega)$ is \textbf{(i)} analytic on every line in ${\mathbb R}^d$, 
	and \textbf{(ii)} differentiable on ${\mathbb R}^d$.
\end{theorem}

\subsection{Derivatives of Eigenvalues}
\subsubsection{First Derivatives of Eigenvalues}\label{subsec:first_der}
Consider a univariate Hermitian matrix-valued function ${\mathcal A}(\omega)$
that depends
on $\omega$ analytically. An analytic eigenvalue $\tilde{\lambda}_j(\omega)$ and the 
associated eigenvector $v_j(\omega)$ as described in Theorem \ref{thm:Rellich}
satisfy
\[
	{\mathcal A}(\omega) v_j(\omega) = \tilde{\lambda}_j(\omega) v_j(\omega).
\]
Taking the derivatives of both sides, we obtain
\begin{equation}\label{eq:der_taken}
	\frac{d {\mathcal A}(\omega)}{d \omega} v_j(\omega)  +  {\mathcal A}(\omega) \frac{d v_j(\omega)}{d \omega}
								=
	\frac{d \tilde{\lambda}_j(\omega)}{d \omega} v_j(\omega)  +  \tilde{\lambda}_j(\omega) \frac{d v_j(\omega)}{d \omega}.
\end{equation}
Multiplying both sides by $v_j(\omega)^{\ast}$ and using the identities
	$v_j(\omega)^{\ast} {\mathcal A}(\omega) = v_j(\omega)^{\ast} \tilde{\lambda}_j(\omega)$
as well as
	$v_j(\omega)^{\ast} v_j(\omega) = \| v_j(\omega) \|^2_2 = 1$,
we get 
\begin{equation}\label{eq:eigval_1der}
	\frac{d \tilde{\lambda}_j(\omega)}{d \omega}
					=
		v_j(\omega)^{\ast}  \frac{d {\mathcal A}(\omega)}{d \omega}   v_j(\omega).
\end{equation}

\subsubsection{Second Derivatives of Eigenvalues}\label{subsec:sec_der}
By differentiating both sides of (\ref{eq:eigval_1der}),  it is possible to deduce the formula
(the details are omitted for brevity)
\small
\begin{equation}\label{eq:eigval_2der}
	\frac{d^2 \tilde{\lambda}_j(\omega)}{d \omega^2}
				=
	v_j(\omega)^{\ast} \frac{d^2 {\mathcal A}(\omega)}{d \omega^2} v_j(\omega)
				+
	2   \sum_{k=1, k\neq j}^n 
				\frac{1}{\tilde{\lambda}_j(\omega) - \tilde{\lambda}_k(\omega)}
		\left| v_k(\omega)^\ast   \frac{d {\mathcal A}(\omega)} {d \omega} v_j(\omega) \right|^2
\end{equation}
\normalsize
for the second derivatives assuming that the (algebraic) multiplicity of $\tilde{\lambda}_j(\omega)$
is one.

If, on the other hand, the eigenvalues repeat at a given $\hat{\omega}$,
specifically when the (algebraic) multiplicity of $\tilde{\lambda}_j(\hat{\omega})$ is greater than 
one, the formula (\ref{eq:eigval_2der}) generalizes as
\small
\begin{equation}\label{eq:eigval_2der_rep}
	\frac{d^2 \tilde{\lambda}_j(\hat{\omega})}{d \omega^2}
				=
	v_j(\hat{\omega})^{\ast} \frac{d^2 {\mathcal A}(\hat{\omega})}{d \omega^2} v_j(\hat{\omega})
				+
	2   \sum_{k=1, k\neq j, k \notin \alpha}^n
		\lim_{\tilde{\omega} \rightarrow \hat{\omega}} 
		\left(
				\frac{1}{\tilde{\lambda}_j(\tilde{\omega}) - \tilde{\lambda}_k(\tilde{\omega})}
		\left| v_k(\tilde{\omega})^\ast   \frac{d {\mathcal A}(\tilde{\omega})} {d \omega} v_j(\tilde{\omega}) \right|^2
		\right).
\end{equation}
\normalsize
Here, $\alpha$ denotes the set of indices of the analytic eigenvalues (specified
in 
Theorem \ref{thm:Rellich}) that are identical to $\tilde{\lambda}_j(\omega)$ at all $\omega$.

\subsubsection{Derivatives of Eigenvalues for Multivariate Hermitian Matrix Functions}\label{subsec:multi_der}
Let ${\mathcal A}(\omega) : {\mathbb R}^d \rightarrow {\mathbb C}^{n\times n}$ be Hermitian and analytic.
It follows from (\ref{eq:eigval_1der}) that 
\begin{equation}\label{eq:eigval_part_der}
	\frac{\partial \tilde{\lambda}_j(\omega)}{\partial \omega_k}
			=	
	v_j^{\ast}(\omega) \frac{\partial {\mathcal A}(\omega)}{\partial \omega_k} v_j(\omega).
\end{equation}
Since ${\mathcal A}(\omega)$ and $v_j(\omega)$ are analytic with respect to $\omega_\ell$ for $\ell = 1,\dots,n$, 
this implies the continuity, also the analyticity with respect to $\omega_\ell$, of each partial derivative 
$\partial \tilde{\lambda}_j(\omega)/\partial \omega_k$, and hence the existence of  $\partial^2
\tilde{\lambda}_j(\omega)/(\partial \omega_k \partial \omega_\ell)$,
everywhere. If the multiplicity of $\tilde{\lambda}_j(\omega)$ is one, differentiating 
both sides of (\ref{eq:eigval_part_der}) with respect to $\omega_\ell$ would yield the following 
expressions for the second partial derivatives.
\small
\begin{eqnarray*}
	\frac{\partial^2 \tilde{\lambda}_j(\omega)}{\partial \omega_k \; \partial \omega_\ell}
						= &
		v_j^{\ast}(\omega)	\frac{\partial^2 {\mathcal A}(\omega)}{\partial \omega_k \; \partial \omega_l} v_j(\omega)
							+  \hskip 58ex \\
						&
		2\cdot \Re
			\left(
					\sum_{m=1, m\neq j}^n	
						\frac{1}{\tilde{\lambda}_j(\omega) - \tilde{\lambda}_m(\omega)}  
						\left(
							v_j(\omega)^{\ast} \frac{\partial {\mathcal A}(\omega)}{\partial \omega_k} v_m(\omega)
						\right) 
						\left(
							v_m(\omega)^{\ast} \frac{\partial {\mathcal A}(\omega)}{\partial \omega_\ell} v_j(\omega)
						\right)
			\right).
\end{eqnarray*}
\normalsize
Expressions similar to (\ref{eq:eigval_2der_rep}) can be obtained for the second partial derivatives 
when $\tilde{\lambda}_j(\omega)$ has multiplicity greater than one.

\subsection{Analyticity of Singular Values}\label{sec:anal_sing_vals}
Some of the applications (see Section \ref{sec:app_dyn_sys}) concern the
optimization 
of the $j$th largest singular value of an analytic matrix-valued function. The singular value problems 
are special Hermitian eigenvalue problems. In particular, denoting the $j$th largest singular 
value of an analytic matrix-valued function ${\mathcal B}(\omega) : {\mathbb R}^d \rightarrow {\mathbb C}^{n\times m}$
(not necessarily Hermitian) by $\sigma_j(\omega)$, the set of eigenvalues of the Hermitian
matrix-valued function 
\[
	{\mathcal A}(\omega)	:=
	\left[
	\begin{array}{cc}
		0						&	{\mathcal B}(\omega) \\
		{\mathcal B}(\omega)^{\ast}	&	0	\\
	\end{array}
	\right],
\]
is $\{ \sigma_j(\omega), -\sigma_j(\omega) : j = 1,\dots, n \}$. In the univariate case 
$\sigma_j(\omega)$ is the $j$th largest of the $2n$ analytic eigenvalues, 
$\tilde{\lambda}_1(\omega),\dots,\tilde{\lambda}_{2n}(\omega)$, of ${\mathcal A}(\omega)$.
The multivariate $d$-dimensional case is similar, with the exception that each 
eigenvalue $\tilde{\lambda}_j(\omega)$ is differentiable and analytic along every line
in ${\mathbb R}^d$. Let us focus on the univariate case throughout the rest of this
section. Extensions to the multi-variate case are similar to the previous
sections.
Suppose 
$v_j(\omega)
			:=
	\left[
		\begin{array}{c}
			u_j(\omega) \\
			w_j(\omega)
		\end{array}
	\right]$,
with $u_j(\omega) \in {\mathbb C}^n$, $w_j(\omega) \in {\mathbb C}^m$, is the analytic eigenvector
function as 
specified in Theorem \ref{thm:Rellich} of ${\mathcal A}(\omega)$ associated 
with $\tilde{\lambda}_j(\omega)$, that is
\[
	\left[
	\begin{array}{cc}
		0						&	{\mathcal B}(\omega) \\
		{\mathcal B}(\omega)^{\ast}	&	0	\\
	\end{array}
	\right]
	\left[
		\begin{array}{c}
			u_j(\omega) \\
			w_j(\omega)
		\end{array}
	\right]
			=
	\tilde{\lambda}_j(\omega)
	\left[
		\begin{array}{c}
			u_j(\omega) \\
			w_j(\omega)
		\end{array}
	\right].
\]
The above equation implies
\begin{equation}\label{eq:sing_val_defn}
	{\mathcal B}(\omega) w_j(\omega) = \tilde{\lambda}_j(\omega) u_j(\omega)
		\;\;\; {\rm and} \;\;\;
	{\mathcal B}(\omega)^{\ast} u_j(\omega) = \tilde{\lambda}_j(\omega) w_j(\omega).
\end{equation}
In other words, $u_j(\omega)$, $w_j(\omega)$ are analytic, and consist of a 
pair of consistent left and right singular vectors associated with $\tilde{\lambda}_j(\omega)$.
To summarize, in the univariate case, $\tilde{\lambda}_j(\omega)$ can be
considered as a signed analytic singular value of ${\mathcal B}(\omega)$, and there is a consistent
pair of analytic left and right singular vector functions, $u_j(\omega)$ and
$w_j(\omega)$, 
respectively.

Next, in the univariate case, we derive expressions for the first derivative of
$\tilde{\lambda}_j(\omega)$, 
in terms of the corresponding left and right singular vectors.  It follows from the singular value equations
(\ref{eq:sing_val_defn}) above that $\| u_j(\omega) \| = \| w_j(\omega) \| = 1/\sqrt{2}$ 
(if $\tilde{\lambda}_j(\omega) = 0$, this equality follows from analyticity).
Now, the application of the expression 
(\ref{eq:eigval_1der}) yields
\begin{eqnarray*}
	\frac{d \tilde{\lambda}_j(\omega)}{d \omega}
					= &
	\left[
		\begin{array}{cc}
			u_j(\omega)^{\ast}	&	w_j(\omega)^{\ast}
		\end{array}
	\right]
	\left[
			\begin{array}{cc}
				0	&	d{\mathcal B}(\omega) / d\omega  \\
				d{\mathcal B}(\omega)^{\ast} / d\omega	&	0 \\
			\end{array}
	\right]
	\left[
		\begin{array}{c}
			u_j(\omega) \\
			w_j(\omega)
		\end{array}
	\right], \\
			= &
	u_j(\omega)^{\ast} \frac{d {\mathcal B}(\omega)}{d \omega} w_j(\omega)
			+
	w_j(\omega)^{\ast} \frac{d {\mathcal B}(\omega)^{\ast}}{d \omega}
u_j(\omega),
			\hskip 18ex \\
			= &
	2\cdot \Re
		\left(
			u_j(\omega)^{\ast} \frac{d {\mathcal B}(\omega)}{d \omega} w_j(\omega)
		\right). \hskip 30ex
\end{eqnarray*}
In terms of the unit left and right singular vectors $\hat{u}_j(\omega) :=
\sqrt{2} \cdot u_j(\omega)$ and 
$\hat{w}_j(\omega) := \sqrt{2} \cdot w_j(\omega)$, respectively, associated
with $\tilde{\lambda}_j(\omega)$, 
we obtain
\begin{equation}\label{eq:first_der_sval}
	\frac{d \tilde{\lambda}_j(\omega)}{d \omega}
				=
	\Re
		\left(
			\hat{u}_j(\omega)^{\ast} \frac{d {\mathcal B}(\omega)}{d \omega} \hat{w}_j(\omega)
		\right).	
\end{equation}

\noindent
\textbf{Notation:} Throughout the rest of the text, we denote the eigenvalues of ${\mathcal A}(\omega)$ 
that are analytic in the univariate case (stated in Theorem \ref{thm:Rellich}), and differentiable and 
analytic along every line in the multivariate case (stated in Theorem \ref{thm:anal_mul_var}) with 
$\tilde{\lambda}_1(\omega), \dots, \tilde{\lambda}_n(\omega)$. On the other hand, $\lambda_j(\omega)$
or $\lambda_j({\mathcal A}(\omega))$ denotes the $j$th largest eigenvalue, and
$\sigma_j(\omega)$ or $\sigma_j({\mathcal A}(\omega))$ denotes the $j$th largest singular value
 of ${\mathcal A}(\omega)$.

\section{Piece-wise Quadratic Support Functions}\label{sec:quad_support}

Let $\tilde{\lambda}_1, \dots, \tilde{\lambda}_n : {\mathbb R}^d \rightarrow {\mathbb R}$ be 
eigenvalue functions of a Hermitian matrix-valued function ${\mathcal A}(\omega)$ that are analytic
along every line in ${\mathbb R}^d$ and differentiable on ${\mathbb R}^d$, and let 
${\mathcal B} \subset {\mathbb R}^d$ be the box defined by
\begin{equation}\label{eq:box}
	{\mathcal B} := {\mathcal B}\left( \omega^{(l)}_1,\omega^{(u)}_1,\dots,\omega^{(l)}_d,\omega^{(u)}_d  \right)
				:=	
				\left\{
					\omega\in {\mathbb R}^d \; | \;
					\omega_j \in \left[ \omega^{(l)}_j, \omega^{(u)}_j \right] \;\; {\rm for} \;\; j=1,\dots,d
				\right\}.
\end{equation}
Consider the closed and connected subsets ${\mathcal P}_1,\dots, {\mathcal P}_q$ of ${\mathcal B}$, 
with $q$ as small as possible, such that $\cup_{k=1}^q {\mathcal P}_k = {\mathcal B}$, and
${\mathcal P}_k \cap {\mathcal P}_\ell = \partial {\mathcal P}_k \cap \partial {\mathcal P}_\ell$ for each 
$k$ and $\ell$, and such that in the interior of ${\mathcal P}_k$ none of the eigenvalue functions 
$\tilde{\lambda}_1, \dots, \tilde{\lambda}_n$ intersect each other. Define 
$\lambda : {\mathcal B} \rightarrow {\mathbb R}$ as follows:
\begin{equation}\label{eq:eigval_defn}
	\lambda(\omega)	:=	
			f\left(
					\tilde{\lambda}_{s_{k1}}(\omega),	\dots,	\tilde{\lambda}_{s_{kj}}(\omega)
			  \right) \;\; {\rm for} \; {\rm all} \; \omega\in {\rm int}\left( {\mathcal P}_k \right)
\end{equation}
where $f$ is analytic, and
$s_k = 
[
\begin{array}{ccc}
	s_{k1} & \dots & s_{kj}
\end{array}
]^T \in {\mathbb Z}_+^j
$ is a vector of indices such that
\[
	\tilde{\lambda}_{s_{ki}}(\omega)	=    \tilde{\lambda}_{s_{\ell i}} (\omega)	\;\; {\rm for} \; i = 1,\ldots,j,
\]
and for all $\omega \in \partial {\mathcal P}_k \cap \partial {\mathcal P}_\ell$ in order to ensure the
continuity of $\lambda(\omega)$ on ${\mathcal B}$. The extremal eigenvalue function
$\lambda(\omega) = \sum_{k = 1}^j   d_k \lambda_{k} (\omega)$ fits into the framework.

We derive a piece-wise quadratic support function $q_k(\omega)$ about a given
point $\omega_k \in {\mathcal B}$ bounding $\lambda(\omega)$ from below
for all $\omega\in {\mathcal B}$, and such that $q_k(\omega_k) = \lambda(\omega_k)$. Let us focus on 
the direction $p := (\omega - \omega_k)/\| \omega - \omega_k \|$, the univariate function 
$\phi(\alpha) := \lambda(\omega_k + \alpha p)$, and the analytic univariate functions 
$\tilde{\phi}_j(\alpha) := \tilde{\lambda}_j(\omega_k + \alpha p)$ for $j = 1,\dots,n$. Also, let us denote 
the isolated points in the interval $\left[ 0, \| \omega - \omega_k \| \right]$, where 
two distinct functions among $\tilde{\phi}_1(\alpha),\dots,\tilde{\phi}_n(\alpha)$ intersect each other by 
$\alpha^{(1)}, \dots, \alpha^{(m)}$. At these points, $\phi(\alpha)$ may not be differentiable. We have
\begin{equation}\label{eq:quad_der_mult_int}
	\lambda(\omega) = \lambda(\omega_k) + \sum_{\ell = 0}^m
\int_{\alpha^{(\ell)}}^{\alpha^{(\ell+1)}}  \phi'(t) dt,
\end{equation}
where $\alpha^{(0)} := 0$ and $\alpha^{(m+1)} := \| \omega - \omega_k \|$. Due to the existence
of the second partial derivatives of $\tilde{\lambda}_j(\omega)$ (since the expression (\ref{eq:eigval_part_der})
implies the analyticity of the first partial derivatives with respect to each parameter disjointly), 
there exists a constant $\gamma$ that satisfies 
\begin{equation}\label{eq:lb_sec_der}
	\lambda_{\min} \left( \nabla^2 \tilde{\lambda}_j (\omega) \right) \geq
\gamma, \quad \textrm{for all}~\omega\in {\mathcal B}, \quad j = 1,\dots,n.
\end{equation}
%for $j = 1,\dots,n$, and for all $\omega\in {\mathcal B}$. 
Furthermore, 
$
	\tilde{\phi}_j''(\alpha) = p^T \; \nabla^2 \tilde{\lambda}_j (\omega_k + \alpha p) \; p \geq 
				\lambda_{\rm min} \left( \nabla^2 \tilde{\lambda}_j (\omega_k +\alpha p) \right) \geq \gamma$
for all $\alpha\in \left[ 0, \| \omega - \omega_k \| \right]$. Thus, applying the mean value theorem to the analytic
functions $\tilde{\phi}'_j(\alpha)$ for $j=1, \dots,n$ and since $\phi'(t) \geq {\rm min}_{j=1,\ldots,n} \tilde{\phi}'_j(t)$, 
we obtain
\[
	\phi'(t) \geq {\rm min}_{j=1,\ldots,n} \tilde{\phi}'_j(0) + \gamma t.
\]
By substituting the last inequality in (\ref{eq:quad_der_mult_int}), 
integrating the right-hand side of (\ref{eq:quad_der_mult_int}), and using 
$\tilde{\phi}_j'(0) = \nabla \tilde{\lambda}_j(\omega_k)^T p$ (since
$\tilde{\lambda}_j$ is differentiable), we arrive at the following:
\begin{theorem}
Suppose ${\mathcal A}(\omega) : {\mathbb R}^d \rightarrow {\mathbb C}^{n\times n}$ is an analytic and Hermitian
matrix-valued function, the eigenvalue function $\lambda(\omega)$ is defined as in (\ref{eq:eigval_defn}) in terms 
of the eigenvalues $\tilde{\lambda}_1(\omega), \dots, \tilde{\lambda}_n(\omega)$ of ${\mathcal A}(\omega)$ that are
differentiable and analytic on every line in ${\mathbb R}^d$,
and $\gamma$ is a lower bound as in (\ref{eq:lb_sec_der}). Then the following inequality holds for all $\omega \in {\mathcal B}$:
\begin{equation}\label{eq:quad_model_multi}
	\lambda(\omega)	\;\; 	\geq		\;\;
	\left[
		q_k(\omega)	:=	
			\lambda(\omega_k)  +  \left( {\rm min}_{j=1,\ldots,n}
				\nabla \tilde{\lambda}_j(\omega_k)^T (\omega - \omega_k) \right)  +  \frac{\gamma}{2} \| \omega - \omega_k \|^2
	\right].
\end{equation}
\end{theorem}

% \newpage

\section{Simplified Piece-wise Quadratic Support Functions}\label{sec:squad_support}

\subsection{Support Functions under Generic Simplicity}\label{subsec:support_analyticity}
In various instances, the eigenvalue functions $\tilde{\lambda}_1, \dots, \tilde{\lambda}_n$
do not intersect each other at any $\omega \in {\mathcal B}$ generically. In
such cases, for some $j$,
we have $\lambda(\omega) = \tilde{\lambda}_j(\omega)$ for all $\omega \in {\mathcal B}$, therefore
$\lambda(\omega)$ is analytic in the univariate case and analytic along every line in the multivariate
case. For instance, the singular values of the matrix function ${\mathcal A}(\omega) := C(\omega i I - A)^{-1} B + D$
involved in the definition of the ${\rm H}_\infty$-norm do not coalesce at any $\omega
\in {\mathbb R}$ on a dense
subset of the set of quadruples $(A,B,C,D)$. Similar remarks apply to all of the specific eigenvalue optimization problems
in Section \ref{sec:app_dyn_sys}.

Under the generic simplicity assumption, the piece-wise quadratic support function (\ref{eq:quad_model_multi})
simplifies to
\begin{equation}\label{eq:quad_simple_multi}
	q_k(\omega)	=	\lambda(\omega_k)	+ \nabla \lambda(\omega_k)^T (\omega - \omega_k) + \frac{\gamma}{2} \| \omega - \omega_k \|^2.
\end{equation}
Here, $\gamma$ is a lower bound on $\lambda_{\min} \left( \nabla^2 \lambda
(\omega) \right)$ for all $\omega \in {\cal B}$.
In many cases, it may be possible to obtain a rough lower bound $\gamma$
numerically by means of the
expressions for the second  derivatives in Sections \ref{subsec:sec_der} and
\ref{subsec:multi_der},
and exploiting the Lipschitz continuity of the eigenvalue $\lambda(\omega)$ and other eigenvalues.

\subsection{Support Functions for Extremal Eigenvalues}\label{subsec:support_extreme}
Consider the extremal eigenvalue function 
\begin{equation}\label{eq:extreme_eig}
	\lambda(\omega) = \sum_{k=1}^j d_k \lambda_k(\omega)
\end{equation}
for given real numbers $d_1 \geq d_2 \geq \dots \geq d_j \geq 0$. A special case when
$d_1, d_2, \dots  d_j$ are integers is discussed in Section \ref{sec:app_min_sum_largest}.
When $d_1 = 1, d_2 = \dots = d_j = 0$, this reduces to the maximal eigenvalue function 
$\lambda(\omega) = \lambda_1(\omega)$ in Section \ref{sec:app_min_largest}. 

For simplicity, let us suppose that $\lambda(\omega)$ is differentiable at $\omega_k$, 
about which we derive a support function below. This is generically  the case. In the unlikely case 
of two eigenvalues coalescing at $\omega_k$, the non-differentiability is isolated at this point. 
Therefore, $\lambda(\omega)$ is differentiable at all nearby points. For a fixed $\omega \in {\mathcal B}$, 
as in the previous section, define $\phi : {\mathbb R} \rightarrow {\mathbb R}, \; \phi(\alpha)  :=  \lambda(\omega_k + \alpha p)$ 
for $p = (\omega - \omega_k) / \| \omega - \omega_k \|$.  Denote the points $\alpha \in (0,\| \omega - \omega_k \|]$ 
where either one of $\lambda_1(\omega_k + \alpha p), \dots, \lambda_j(\omega_k + \alpha p)$ 
is not simple by $\alpha^{(1)}, \dots, \alpha^{(m)}$ in increasing order. These are the points where
$\phi(\alpha) $ is possibly not analytic. At any $\alpha \in (0,\| \omega - \omega_k \|)$, by $\phi_{+,\alpha}$
we refer to the analytic function satisfying $\phi_{+,\alpha}(\tilde{\alpha}) = \phi(\tilde{\alpha})$ for all 
$\tilde{\alpha} > \alpha$ sufficiently close to $\alpha$. Similarly, $\phi_{-,\alpha}$ refers to the analytic 
function satisfying  $\phi_{-,\alpha}(\tilde{\alpha}) = \phi(\tilde{\alpha})$ for all  $\tilde{\alpha} < \alpha$
sufficiently close to $\alpha$. Furthermore, $\phi'_+(\alpha)$ and $\phi'_-(\alpha)$ represent the right-hand and left-hand
derivatives of $\phi(\alpha)$, respectively.
\begin{lemma}\label{lemma:left_right_pieces}
The following relation holds for all $\alpha \in (0,\| \omega - \omega_k \|)$: 
	\textbf{(i)} $\phi_{+,\alpha}(\tilde{\alpha}) \geq \phi_{{-},\alpha}(\tilde{\alpha})$ for all  $\tilde{\alpha} > \alpha$
	sufficiently close to $\alpha$;
	\textbf{(ii)} consequently $\phi'_+(\alpha) \geq \phi'_{-}(\alpha)$.
\end{lemma}
\begin{proof}
The functions $\phi_{-,\alpha}(\tilde{\alpha})$ and $\phi_{+,\alpha}(\tilde{\alpha})$ are of the form
\begin{equation}\label{eq:leftright_pieces}
	\phi_{-,\alpha}(\tilde{\alpha})	=	\sum_{k=1}^j d_k \tilde{\lambda}_{n_k}(\omega_k + \tilde{\alpha} p)
				\;\;\;	{\rm and}	\;\;\;
	\phi_{+,\alpha}(\tilde{\alpha})	=	\sum_{k=1}^j d_k \lambda_{k}(\omega_k + \tilde{\alpha} p)
\end{equation}
for some indices  $n_1, \dots, n_j$ and for all $\tilde{\alpha}  >  \alpha$ sufficiently close to $\alpha$. 
In (\ref{eq:leftright_pieces}), the latter equality follows from $\phi_{+,\alpha}(\tilde{\alpha}) = \phi(\tilde{\alpha}) = \lambda(\omega_k + \tilde{\alpha} p)$
for all $\tilde{\alpha} > \alpha$ sufficiently close to $\alpha$ by definition. The former equality is
due to $\phi_{-,\alpha}(\tilde{\alpha}) = \phi(\tilde{\alpha}) = \lambda(\omega_k + \tilde{\alpha} p)$
for all $\tilde{\alpha} < \alpha$ sufficiently close to $\alpha$ implying $\phi_{-,\alpha}(\tilde{\alpha})$
is a weighted sum of $j$ of the analytic eigenvalues 
$\tilde{\lambda}_{1}(\omega_k + \tilde{\alpha} p), \dots,  \tilde{\lambda}_{n}(\omega_k + \tilde{\alpha} p)$
with weights $d_1, \dots, d_j$.
We rephrase the inequality $\phi_{+,\alpha}(\tilde{\alpha})  \geq \phi_{-,\alpha}(\tilde{\alpha})$ as
$
	\phi_{+,\alpha}(\tilde{\alpha})	-	\phi_{-,\alpha}(\tilde{\alpha})	
							= 
	\sum_{k=1}^j d_k \cdot a_k  \geq 0,						
$
where $a_k =  \lambda_{k}(\omega_k + \tilde{\alpha} p)	-	 \tilde{\lambda}_{n_k}(\omega_k + \tilde{\alpha} p)$, where $ k = 1,\ldots,j$.

Note that $\sum_{k=1}^q a_k \geq 0$ for each $q = 1,\ldots,j$ since
$\lambda_{1}(\omega_k + \tilde{\alpha} p), \dots,  \lambda_{q}(\omega_k +
\tilde{\alpha} p)$
are the largest $q$ eigenvalues. In particular, their sum cannot be less than
the sum of
$\tilde{\lambda}_{n_1}(\omega_k + \tilde{\alpha} p), \dots,
\tilde{\lambda}_{n_q}(\omega_k + \tilde{\alpha} p)$.
Therefore,
\begin{eqnarray*}
\sum_{k=1}^{j} d_k \cdot a_k & = & d_j \sum_{k=1}^j a_k + (d_{j-1} - d_j) \sum_{k=1}^{j-1} a_k + (d_{j-2} - d_{j-1}) \sum_{k=1}^{j-2} a_k + \ldots \\
 & & \quad \quad + (d_1 - d_2) a_1, \\ 
 &\geq & 0, 
\end{eqnarray*}
since $d_1 \geq d_2 \geq \ldots \geq d_j \geq 0$.

Part \textbf{(ii)} is immediate from part \textbf{(i)} due to 
$\phi'_+(\alpha)  =  \phi'_{+,\alpha}(\alpha) \geq \phi'_{-,\alpha}(\alpha) = \phi'_-(\alpha)$,  
where the inequality follows from an application of the Taylor's theorem to $\phi_{+,\alpha}(\tilde{\alpha})$
and $\phi_{-,\alpha}(\tilde{\alpha})$ for $\tilde{\alpha} > \alpha$ and around $\alpha$.
\end{proof}

\begin{theorem}\label{thm:simplified_support}
Let $\omega_k \in {\mathbb R}^d$ be such that all of the eigenvalues
$\lambda_1(\omega_k), \dots, \lambda_j(\omega_k)$ are simple, and let $\gamma$
satisfy $\lambda_{\min} \left( \nabla^2 \lambda (\omega) \right) \geq \gamma$ for all 
$\omega \in {\cal B}$ such that $\lambda(\omega)$ is simple. Then, the following inequality 
holds for the extreme eigenvalue function $\lambda(\omega)$ given by (\ref{eq:extreme_eig})
and for all $\omega \in {\mathcal B}$:
\[
	\lambda(\omega) \;\;   \geq  \;\;   
						q_k(\omega)	  :=
						\lambda(\omega_k)   +   \nabla \lambda(\omega_k)^T (\omega - \omega_k)
									+	\frac{\gamma}{2} \| \omega - \omega_k \|^2.
\]
\end{theorem}
\begin{proof}
It follows from the Taylor's theorem that, for each $k = 0, \dots, m$ and $\alpha$ that belongs to the closure
of $(\alpha^{(k)}, \alpha^{(k+1)})$ (here we define $\alpha^{(0)} = 0$ and $\alpha^{(m+1)} = \| \omega - \omega_k \|$), we have
\[
		\phi(\alpha) 	= 		\phi(\alpha^{(k)})		+		\phi'_+(\alpha^{(k)}) (\alpha  -  \alpha^{(k)})
											+		\frac{ \phi''(\eta) }{2} (\alpha - \alpha^{(k)})^2,
\]
where $\eta \in (\alpha^{(k)}, \alpha)$. By observing
$
	 \phi''(\eta) 	=		p^T \nabla^2 \lambda (\omega_k + \eta p) p		\geq		\lambda_{\min} \left[  \nabla^2(\lambda(\omega_k + \eta p))  \right] 	\geq \gamma,
$
we deduce
\begin{equation}\label{eq:simp_support_ineq1}
		\phi(\alpha) 	\geq		\phi(\alpha^{(k)})		+		\phi'_+(\alpha^{(k)}) (\alpha  -  \alpha^{(k)})	+	\frac{\gamma}{2}	(\alpha  -  \alpha^{(k)})^2,
\end{equation}
and by differentiating we also deduce
\begin{equation}\label{eq:simp_support_ineq2}
  \phi'_{-}(\alpha)	\;\;  \geq \;\;		\phi'_+(\alpha^{(k)})		+		\gamma (\alpha  -  \alpha^{(k)}).
\end{equation}

Next, we prove the inequalities
\begin{eqnarray}
\label{eq;simp_support_ineq3}
\phi(\alpha^{(\ell)}) 		\;\;	\geq	\;\;	&	\phi(\alpha^{(0)})		+		\phi'_+(\alpha^{(0)}) (\alpha^{(\ell)}  -  \alpha^{(0)})	+	\frac{\gamma}{2}	(\alpha^{(\ell)}  -  \alpha^{(0)})^2, \\
\label{eq;simp_support_ineq4}
\phi'_{-}(\alpha^{(\ell)}) 		\geq	\;\;	&	\phi'_+(\alpha^{(0)})		+		\gamma (\alpha^{(\ell)}  -  \alpha^{(0)}),		\hskip 23ex
\end{eqnarray}
for each $\ell = 1, 2, \dots, m+1$ by induction. The inequalities (\ref{eq;simp_support_ineq3}) and (\ref{eq;simp_support_ineq4}) for $\ell = 1$ 
hold by applications of equations (\ref{eq:simp_support_ineq1}) and (\ref{eq:simp_support_ineq2}) with $\alpha = \alpha^{(1)}, \alpha^{(k)} = \alpha^{(0)}$. 
Let us suppose that the inequalities indeed hold for $\ell = 2, \dots, k$ as the inductive hypothesis.
Now, by another application of (\ref{eq:simp_support_ineq2}) and since $\phi'_+(\alpha^{(k)}) \geq \phi'_{-}(\alpha^{(k)})$ (see Lemma \ref{lemma:left_right_pieces}), 
we obtain
\begin{eqnarray*}
	\phi'_{-}(\alpha^{(k+1)}) 	\;	\geq	\;\; &	 \phi'_+(\alpha^{(k)})		+		\gamma (\alpha^{(k+1)}  -  \alpha^{(k)}),	\hskip 16ex \\
							\geq \;\; & \phi'_-(\alpha^{(k)})		+		\gamma (\alpha^{(k+1)}  -  \alpha^{(k)}),	\hskip 16ex \\
							\geq  \;\; &   \left[		\phi'_+(\alpha^{(0)})		+		\gamma (\alpha^{(k)}  -  \alpha^{(0)})		\right]
												+		\gamma (\alpha^{(k+1)}  -  \alpha^{(k)}), \\
							= \;\; 	&	\phi'_+(\alpha^{(0)})		+		\gamma (\alpha^{(k+1)}  -  \alpha^{(0)}), \hskip 16ex
\end{eqnarray*}
where we use the inductive hypothesis with $\ell = k$ in the third inequality. Furthermore, by (\ref{eq:simp_support_ineq1}), the inequality
$\phi'_+(\alpha^{(k)}) \geq \phi'_{-}(\alpha^{(k)})$, and exploiting the
inductive hypothesis with $\ell = k$, we end up with
\begin{eqnarray*}
	\phi(\alpha^{(k+1)}) 		\;\;	\geq	&	\phi(\alpha^{(k)})		+		\phi'_+(\alpha^{(k)}) (\alpha^{(k+1)}  -  \alpha^{(k)})	+	\frac{\gamma}{2}	(\alpha^{(k+1)}  -  \alpha^{(k)})^2, 	\hskip 8ex \\	
							\geq & 	\phi(\alpha^{(k)})		+		\phi'_-(\alpha^{(k)}) (\alpha^{(k+1)}  -  \alpha^{(k)})	+	\frac{\gamma}{2}	(\alpha^{(k+1)}  -  \alpha^{(k)})^2, 	\hskip 8ex \\
							=	&	
					\left[		\phi(\alpha^{(0)})		+		\phi'_+(\alpha^{(0)}) (\alpha^{(k)}  -  \alpha^{(0)})	+	\frac{\gamma}{2}	(\alpha^{(k)}  -  \alpha^{(0)})^2		\right]	\hskip 12ex \\
							&
	+		\left[\phi'_+(\alpha^{(0)})		+		\gamma (\alpha^{(k)}  -  \alpha^{(0)}) \right] (\alpha^{(k+1)}  -  \alpha^{(k)})	+	\frac{\gamma}{2}	(\alpha^{(k+1)}  -  \alpha^{(k)})^2, \\		
							= &	\;\;	\phi(\alpha^{(0)})		+		\phi'_+(\alpha^{(0)}) (\alpha^{(k+1)}  -  \alpha^{(0)})	+	\frac{\gamma}{2}	(\alpha^{(k+1)}  -  \alpha^{(0)})^2,		\hskip 9ex		
\end{eqnarray*}
proving the validity of (\ref{eq;simp_support_ineq3}) and (\ref{eq;simp_support_ineq4}) for each $\ell = 1, \dots, m+1$.

The inequality (\ref{eq;simp_support_ineq3}) with $\ell = m+1$  yields
\[
	\phi(\alpha^{(m+1)}) 		\;\;	\geq	\;\;		\phi(\alpha^{(0)})		+		\phi'_+(\alpha^{(0)}) (\alpha^{(m+1)}  -  \alpha^{(0)})	+	\frac{\gamma}{2}	(\alpha^{(m+1)}  -  \alpha^{(0)})^2,
\]
from which the result follows by noting $\phi(\alpha^{(m+1)}) = \lambda(\omega)$, $\phi(\alpha^{(0)}) = \lambda(\omega_k)$, 
$\phi'_+(\alpha^{(0)}) = \nabla \lambda(\omega_k)^T p$, and $(\alpha^{(m+1)} - \alpha^{(0)}) = \| \omega - \omega_k \|$.
\end{proof}

A lower bound $\gamma$ as stated in Theorem \ref{thm:simplified_support} can be deduced analytically in various 
cases. This is discussed next.

%In some other cases, it may be possible to retrieve a $\gamma$ numerically again using the expressions for the second derivatives 
% of $\lambda(\omega)$ from Subsections \ref{subsec:sec_der} and \ref{subsec:multi_der}.

\section{Lower Bound $\gamma$ for the Extremal Eigenvalue Functions}\label{sec:extreme_low_bound}
The quadratic support functions discussed so far rely on the lower bound $\gamma$. Such bounds can 
be derived analytically for the extremal eigenvalue function (\ref{eq:extreme_eig})
with $d_1 \geq d_2 \geq \dots \geq d_j \geq 0$ in various cases.

\begin{theorem}\label{thm:extreme_low_bound}
Suppose ${\mathcal A}(\omega) : {\mathbb R}^d \rightarrow {\mathbb C}^{n\times n}$ is analytic and
Hermitian at all $\omega \in {\mathbb R}^d$. Then, $\lambda(\omega)$ defined as in (\ref{eq:extreme_eig})
satisfies
\[
	\lambda_{\min} \left( \nabla^2 \lambda (\omega) \right)	\geq \left( 	\sum_{k=1}^j d_k \right) \cdot  \lambda_{\min} \left(	\nabla^2{\mathcal A} (\omega)	\right)
\]
for each $\omega$ such that $\lambda_1(\omega), \dots, \lambda_j(\omega)$ are simple,
where $\nabla^2{\mathcal A}(\omega) \in {\mathbb C}^{nd\times nd}$ is given by
\[
	\nabla^2 {\mathcal A}(\omega)
				=
	\left[
			\begin{array}{cccc}
				\frac{\partial^2 {\mathcal A}(\omega)}{\partial \omega_1^2}
						&
				\frac{\partial^2 {\mathcal A}(\omega)}{\partial \omega_1 \partial \omega_2}
						&
					\dots
						&
				\frac{\partial^2 {\mathcal A}(\omega)}{\partial \omega_1 \partial \omega_d}		\\ 
				\frac{\partial^2 {\mathcal A}(\omega)}{\partial \omega_2 \partial \omega_1}
						&
				\frac{\partial^2 {\mathcal A}(\omega)}{\partial \omega_2^2}
						&
					\dots
						&
				\frac{\partial^2 {\mathcal A}(\omega)}{\partial \omega_2 \partial \omega_d}		\\
						&
						&
						\ddots
						&												\\
				\frac{\partial^2 {\mathcal A}(\omega)}{\partial \omega_d \partial \omega_1}
						&
				\frac{\partial^2 {\mathcal A}(\omega)}{\partial \omega_d \omega_2}
						&
					\dots
						&
				\frac{\partial^2 {\mathcal A}(\omega)}{\partial \omega_d^2}
			\end{array}
	\right].
\]
\end{theorem}
\begin{proof}
First observe that, by the formulas in Section \ref{subsec:multi_der}, we obtain
\begin{eqnarray*}
	\frac{\partial^2 \lambda(\omega)}{\partial \omega_\ell \; \partial \omega_i}
						& = 
		\sum_{k=1}^j
		d_k \cdot
		\left(
		v_k(\omega)^\ast	\frac{\partial^2 {\mathcal A}(\omega)}{\partial \omega_\ell \; \partial \omega_i} v_k(\omega)
		\right)
							+  \hskip 48ex \\
						&
		2\cdot \Re \left[ 
		\sum_{k=1}^j
			\left(
					\sum_{m=k+1}^n	
						\frac{d_k - d_m}{\lambda_k(\omega) - \lambda_m(\omega)}  
						\left(
							v_k(\omega)^\ast \frac{\partial {\mathcal A}(\omega)}{\partial \omega_\ell} v_m(\omega)
						\right) 
						\left(
							v_m(\omega)^\ast \frac{\partial {\mathcal A}(\omega)}{\partial \omega_i} v_k(\omega)
						\right)
			\right) \right],
\end{eqnarray*} 
where we define $d_{j+1} = \dots = d_n = 0$.
As for the Hessian, this yields
\begin{equation}\label{eq:ext_Hessian_inter}
	\nabla^2 \lambda (\omega)
				=
	\sum_{k=1}^j d_k {\mathcal H}^{(k)}(\omega)
				+
	2	\cdot
	\sum_{k=1}^j	\sum_{m= k+1}^n \frac{d_k - d_m}{\lambda_k(\omega) -
\lambda_m(\omega)}   \Re \left( {\mathcal H}^{(k,m)}(\omega) \right),
\end{equation}
where ${\mathcal H}^{(k)}(\omega), {\mathcal H}^{(k,m)}(\omega) \in {\mathbb R}^{d\times d}$ are such that the $(\ell, i)$ entries 
of ${\mathcal H}^{(k)}(\omega)$ and ${\mathcal H}^{(k,m)}(\omega)$ are given by
\[
	v_k(\omega)^\ast	\frac{\partial^2 {\mathcal A}(\omega)}{\partial \omega_\ell \; \partial \omega_i} v_k(\omega)
			\;\;	{\rm and}	\;\;
						\left(
							v_k(\omega)^\ast \frac{\partial {\mathcal A}(\omega)}{\partial \omega_\ell} v_m(\omega)
						\right) 
						\left(
							v_m(\omega)^\ast \frac{\partial {\mathcal A}(\omega)}{\partial \omega_i} v_k(\omega)
						\right),
\]
respectively. Furthermore, it can easily be verified that ${\mathcal H}^{(k,m)}(\omega)$ is positive semi-definite
for each $k$ and $m$. Indeed, denoting
$
	h^{(k,m)}_i := v_m(\omega)^\ast  ( \partial {\mathcal A}(\omega) / \partial \omega_i ) v_k(\omega),
$
for each $u \in {\mathbb C}^d$, observe that
\[
	u^T	{\mathcal H}^{(k,m)}(\omega) u	=	\left|		\sum_{i = 1}^d u_i  h^{(k,m)}_i		\right|^2 \geq 0.
\]
This in turn implies that $\Re \left( {\mathcal H}^{(k,m)}(\omega) \right)$ is positive semi-definite, i.e., 
since ${\mathcal H}^{(k,m)}(\omega) = \Re \left( {\mathcal H}^{(k,m)}(\omega) \right) + i \Im \left( {\mathcal H}^{(k,m)}(\omega) \right)$, for all
$u \in {\mathbb R}^d$ we have $u^T \Re \left( {\mathcal H}^{(k,m)}(\omega) \right) u = u^T {\mathcal H}^{(k,m)}(\omega) u$ $ \geq 0$.
Consequently, it follows from (\ref{eq:ext_Hessian_inter}) that 
\begin{eqnarray*}
	\lambda_{\min} \left(	 	\nabla^2 \lambda (\omega)	\right)
				\geq
	\lambda_{\min} \left(		\sum_{k=1}^j d_k {\mathcal H}^{(k)}(\omega)	\right)
				&
				\geq
				\sum_{k=1}^j d_k \lambda_{\min}  \left(   {\mathcal H}^{(k)}(\omega)  \right),	\\
				&
				\geq
				\sum_{k=1}^j d_k \lambda_{\min}  \left( \nabla^2{\mathcal A}(\omega) \right).
\end{eqnarray*}
To see the last inequality, note that 
${\mathcal H}^{(k)}(\omega) =  \left[ I_d \otimes v_k(\omega)^\ast \right] \cdot \nabla^2 {\mathcal A}(\omega) \cdot \left[ I_d \otimes v_k(\omega) \right]$,
where $\otimes$ denotes the Kronecker product, and therefore for some $v(\omega) \in {\mathbb C}^d$ of unit length
\begin{eqnarray*}
	\lambda_{\min}  \left(   {\mathcal H}^{(k)}(\omega)  \right)
						 =
		v(\omega)^\ast \cdot {\mathcal H}^{(k)}(\omega) \cdot v(\omega) 
						& =
	\left[ v(\omega)^\ast \otimes v_k(\omega)^\ast \right] \cdot  \nabla^2 {\mathcal A}(\omega) \cdot \left[ v(\omega) \otimes v_k(\omega) \right] \\
						& \geq
	\lambda_{\min}	  \left(   \nabla^2 {\mathcal A}(\omega)	\right). \hskip 26ex
\end{eqnarray*}
\end{proof}

There are various instances when the theorem above reveals an immediate lower bound $\gamma$
on $\lambda_{\min} \left( \nabla^2 \lambda (\omega) \right)$ over all $\omega \in {\cal B}$.
A particular instance is given by the corollary below when ${\mathcal A}(\omega)$ is a quadratic 
function of $\omega$.
\begin{corollary}\label{cor:extreme_quad_low_bound}
Suppose that ${\mathcal A}(\omega) : {\mathbb R}^d \rightarrow {\mathbb C}^{n\times n}$ is of the
form
\[
	{\mathcal A}(\omega)
			=
	A_0	+ \sum_{\ell =1}^d \omega_\ell A_\ell	+  \frac{1}{2} \sum_{\ell=1}^d \sum_{i=1}^d \omega_\ell \; \omega_i A_{\ell i}
\]
where $A_\ell$ and $A_{\ell i}$ are Hermitian and $A_{\ell i} = A_{i \ell}$. Then $\lambda(\omega)$ defined 
as in (\ref{eq:extreme_eig}) satisfies
\[
	\lambda_{\min} \left( \nabla^2 \lambda (\omega) \right)
				\geq
	\left( 	\sum_{k=1}^j d_k \right) \cdot  \lambda_{\min} 
		\left(		
			\left[
			\begin{array}{cccc}
					A_{11}
						&
					A_{12}
						&
					\dots
						&
					A_{1d}		\\ 
					A_{21}
						&
					A_{22}
						&
					\dots
						&
					A_{2d}		\\
						&
						&
						\ddots
						&		\\
					A_{d1}
						&
					A_{d2}
						&
					\dots
						&
					A_{dd}
			\end{array}
			\right]	
			\right)
\]
for each $\omega$ such that $\lambda_1(\omega), \dots, \lambda_j(\omega)$ are simple.
\end{corollary}

\section{The Algorithm}\label{sec:algorithm}
We consider the minimization of an eigenvalue function $\lambda : {\mathcal B} \rightarrow {\mathbb R}$ 
formally defined in Section \ref{sec:quad_support} over a box ${\mathcal B} \subset {\mathbb R}^d$. 
We will utilize the quadratic function given by (\ref{eq:quad_simple_multi}) under the assumption that this is 
indeed a support function. This is certainly the case for $\lambda(\omega) = \sum_{k=1}^j d_k \lambda_k(\omega)$ 
with weights $d_k$ such that $d_1 \geq d_2 \geq \dots \geq d_j \geq 0$ as discussed in Section \ref{subsec:support_extreme}
(specifically see Theorem \ref{thm:simplified_support}), and generically the case for all eigenvalue or singular
value functions that arise from various applications in Section \ref{sec:app_dyn_sys}, since these eigenvalue and
singular value functions are simple at all $\omega \in {\mathcal B}$ generically, as discussed in 
Section \ref{subsec:support_analyticity}. Thus the algorithm is applicable to all
of the problems in Section \ref{sec:app}.

The algorithm starts by constructing a support function $q_0$ about an arbitrary point
$\omega_0 \in {\mathcal B}$. The next step is to find the global minimizer
$\omega_1$ of $q_0$ in ${\mathcal B}$,
and to construct the quadratic support function $q_1$ about $\omega_1$. In
general, suppose there are $s+1$
support functions $q_0, q_1, \dots, q_s$ about the points $\omega_0, \omega_1, \dots, \omega_s$.
A new quadratic support function $q_{s+1}$ is constructed about $\omega_{s+1}$, which is a global
minimizer of
\begin{equation}\label{eq:piece_quad}
	\overline{q}_s(\omega) = \max_{k=0,\ldots,s} \; q_k(\omega).
\end{equation}

\begin{algorithm}
 \begin{algorithmic}[1]
  \REQUIRE{A box ${\mathcal B} \subset {\mathbb R}^d$ and an eigenvalue function 
  $\lambda: {\mathcal B} \rightarrow {\mathbb R}$ satisfying either the generic simplicity assumption
  or of the form $\lambda(\omega) = \sum_{k=1}^j d_k \lambda_k(\omega)$ with $d_1 \geq  \dots \geq d_j \geq 0$ 
  }
\STATE Pick an arbitrary $\omega_0 \in {\cal B}$. 
\STATE $u_1 \gets \lambda(\omega_0)$; $\omega best \gets \omega_0$.
\STATE $\overline{q}_0(\omega) := q_0 (\omega) := \lambda(\omega_0) +   \nabla
\lambda(\omega_0)^T (\omega - \omega_0) + (\gamma/2) \|\omega - \omega_0\|^2$.
\STATE $\omega_1 \gets \arg \min_{\omega \in {\cal B}} \; \overline{q}_0(\omega);~l_1 \gets \overline{q}_0(\omega_1)$.
\IF{$\lambda(\omega_1) < u_1$}
\STATE $u_1 \gets \lambda(\omega_1); ~\omega best \gets \omega_1$.
\ENDIF
\STATE $s \gets 1$.
\STATE While $u_s - l_s > \epsilon$ do
\LOOP
\STATE $q_{s}(\omega) := \lambda(\omega_s) + \nabla \lambda (\omega_s)^T (\omega - \omega_s)
 + (\gamma/2) \|\omega - \omega_s\|^2$.
\STATE $\overline{q}_s(\omega) := \max_{k=0,\ldots,s}\{q_k(\omega)\}$.
\STATE $\omega_{s+1} \gets \arg \min_{\omega \in {\cal B}} \; \overline{q}_s(\omega);~l_{s+1} \gets \overline{q}_s(\omega_{s+1})$.
\IF{$\lambda(\omega_{s+1}) < u_s$}
\STATE $u_{s+1} \gets \lambda(\omega_{s+1});~\omega best \gets \omega_{s+1}$.
\ELSE 
\STATE $u_{s+1} \gets u_s$.
\ENDIF
\STATE $s \gets s + 1$.
\ENDLOOP
\STATE {\bf Output:} $l_s, u_s, \omega best$.
 \end{algorithmic}
\caption{Support-based Eigenvalue Optimization}
\label{alg1}
\end{algorithm}

The details of the algorithm are formally presented in Algorithm \ref{alg1}. In
the algorithm, the computationally challenging task is the determination of  
a global minimizer of $\; \overline{q}_s$ defined by (\ref{eq:piece_quad}).
For this purpose, we partition 
the box ${\mathcal B}$ into regions ${\mathcal R}_0, \dots, {\mathcal R}_s$ such
that the quadratic function $q_k$ takes the largest value inside 
the region ${\mathcal R}_k$ (see Figure \ref{fig:partition_box}). Therefore, the
minimization of $\overline{q}_s$ over the region 
${\mathcal R}_k$ is equivalent to the minimization of $q_k$ over the same region. 
This problem can be posed as the following quadratic programming
problem:
\begin{equation}\label{eq:quad_program}
	  \begin{array}{l@{\hspace{10pt}}l}
			{\rm minimize}_{\omega\in {\mathbb R}^d}%
	             					& q_k(\omega)  \\[2ex]
		{\rm subject}\; {\rm to}     & q_k(\omega) \geq q_\ell(\omega), \;\; \ell
\neq k,  \\
			  				&  \omega_j  \in 
[\omega^{(l)}_j,\omega^{(u)}_j], \;\;\;\;\; j =1,\dots,d
 	 \end{array}
\end{equation}
Here, we remark that the seemingly quadratic inequalities $q_k(\omega) \geq
q_\ell(\omega)$ are in fact linear since the quadratic terms cancel out.
\begin{figure}
   	\begin{center}
		\includegraphics[height=0.3\textheight,]{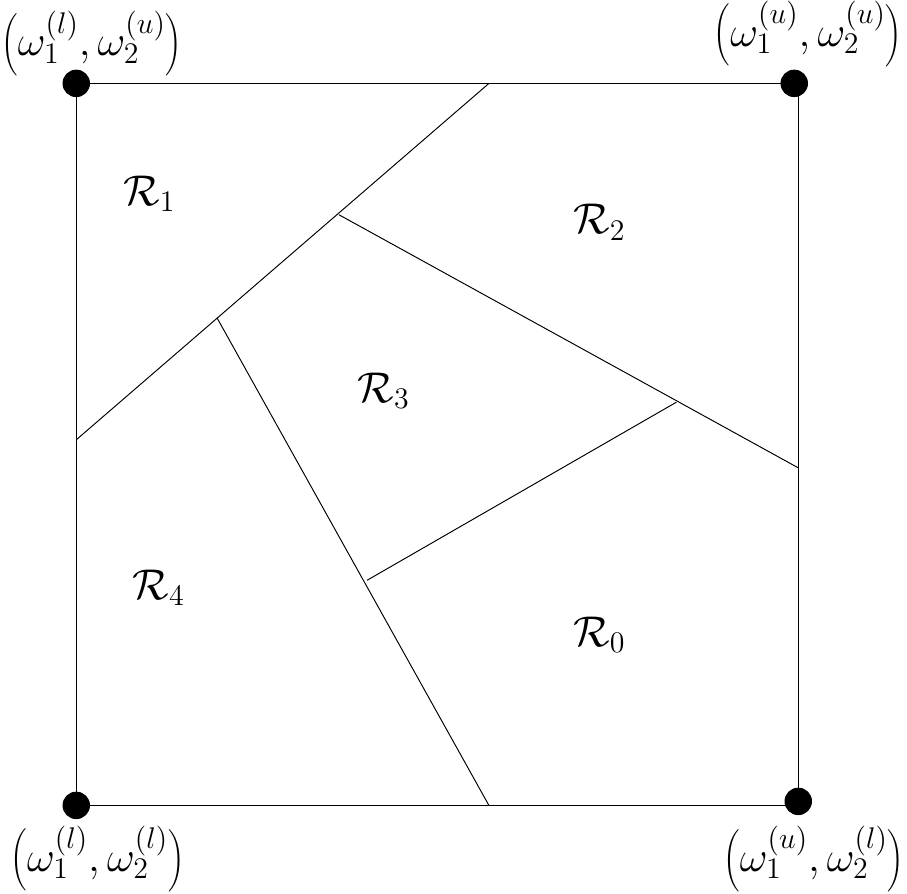}
	\end{center}
	    \caption{ To minimize the piece-wise quadratic function $\overline{q}_s$ over the box
	    ${\mathcal B}$, the box is split into regions ${\mathcal R}_k, \; k=0,\dots,s$ such
	    that $q_k$ is the largest inside the region ${\mathcal R}_k$. Above, a possible
	    partitioning with $s = 4$ is illustrated in the 2-dimensional case. 
}\label{fig:partition_box}
\end{figure}
The feasibility of the algorithm largely relies on the efficiency with which we
can solve
subproblems (\ref{eq:quad_program}). These subproblems are non-convex whenever
$\gamma < 0$. However, since they involve the optimization of a concave function
over a polytope, 
the solution for each subproblem must be attained at one of the vertices of the polytope.

Breiman and Cutler introduced the notion of quadratic support functions of the
form (\ref{eq:quad_simple_multi}) for global optimization 
\cite{Breiman1993}\footnote{We are grateful to an anonymous referee who
pointed us to this reference.}.  They described how
the
subproblems of the form (\ref{eq:quad_program}) can be solved efficiently. At
each iteration,
when a new quadratic support function is added, a new polytope is introduced
associated with it, and some points - call these dead vertices - that used to
be vertices of some polytope are no longer vertices (i.e., they now lie
strictly inside the
new polytope).
A vertex $v$ is dead after the introduction of $q_{s+1}(\omega)$ if and only if
$q_{s+1}(\omega) > \overline{q}_s(\omega)$. Breiman and Cutler observed that these 
dead vertices form a connected graph. Therefore, they can be identified
efficiently. Furthermore, 
a new vertex $v$ (on the boundary of the new polytope) appears between a dead 
vertex $v_d$ and each vertex $v_a$ - each existing vertex that is still a vertex after the
addition of $q_{s+1}(\omega)$ and adjacent to $v_d$ - given by the formula
\begin{eqnarray*}
  	v
		=
	\left[  \frac{q_{s+1} (v_a)   -   \overline{q}_s (v_a)}{ (q_{s+1} (v_a)   -   \overline{q}_s (v_a))   -    (q_{s+1} (v_d)   -   \overline{q}_s (v_d)) } \right]   v_d
		+ \hskip 30ex \\
	\hskip 30ex
	\left[ 1   -   \frac{q_{s+1} (v_a)   -   \overline{q}_s (v_a)}{ (q_{s+1} (v_a)   -   \overline{q}_s (v_a))   -    (q_{s+1} (v_d)   -   \overline{q}_s (v_d)) }   \right] v_a.
\end{eqnarray*}
These new vertices, possibly with the addition of some of the vertices where only
box constraints are active, form the set of vertices for the new polytope. The
edges for the new polytope can be determined by the common active constraints
of its vertices; two vertices are adjacent if and only if there are $d-1$ active 
constraints common to these vertices. All these observations accompanied 
by appropriate data structures (e.g., a heap of vertices indexed based on the values of
$\overline{q}_s$ after the formation of $q_s$ and an adjacency list for each
vertex)
lead to an efficient algorithm when the dimension $d$ is small.
We refer to \cite{Breiman1993} for further details.

Algorithm \ref{alg1} could be based on the more general piece-wise quadratic support 
functions (\ref{eq:quad_model_multi}) by adjusting lines 3 and 11 accordingly. This
would make the algorithm applicable for the optimization of more general eigenvalue
functions, i.e., those that do not involve the sum of the largest eigenvalues 
and violating generic simplicity everywhere. Subproblems analogous to (\ref{eq:quad_program}) could
be devised, however, their solutions appear prohibitively expensive.

\section{Convergence Analysis}\label{sec:1d_anal}

In this section, we analyze the convergence of Algorithm \ref{alg1} - in the general setting
when the support functions (\ref{eq:quad_model_multi}) are used on lines 3 and 11 - for the
following optimization problem:  
\[
\textrm{(P)} \quad \lambda^* := \min_{\omega \in {\cal B}} \lambda(\omega).
\]
Recall that the algorithm starts off by picking an arbitrary point $\omega_0 \in {\cal B}$. At iteration $s$, the algorithm picks 
$\omega_{s+1}$ to be a global minimizer of $\overline{q}_{s}(x)$ over ${\mathcal B}$, where $\overline{q}_{s}(\omega)$ is the 
maximum of the functions $q_k(\omega)$ constructed at the points $\omega_k,~k = 0,\ldots,s$ in ${\cal B}$.  Note that 
$\{l_s\}$ is a non-decreasing sequence of lower bounds on $\lambda^*$, while $\{u_s\}$ is a non-increasing sequence of upper 
bounds on $\lambda^*$. 

We require that $\lambda: {\mathcal B} \rightarrow {\mathbb R}$ be a continuous
and piece-wise function 
defined in terms of the differentiable functions $\tilde{\lambda}_j: {\mathbb R}^d \rightarrow {\mathbb R}$, $j = 1,\ldots,n$
as described in Section \ref{sec:quad_support}. The differentiability of each
$\tilde{\lambda}_j$ on ${\mathbb R}^d$ 
implies the boundedness of $\| \nabla \tilde{\lambda}_j (x) \|$ on ${\mathcal B}$. Consequently, we define
\[
	\mu :=  \max_{j=1,\ldots,n} \max_{\omega \in {\mathcal B}} \| \nabla \tilde{\lambda}_j(\omega) \|.	
\]
We furthermore require each piece $\tilde{\lambda}_j$ to be analytic along every line in ${\mathbb R}^d$,
and exploit the existence of a scalar $\gamma$ that satisfies (\ref{eq:lb_sec_der}). Our convergence analysis 
depends on the scalars $\mu$ and $\gamma$. We now establish the convergence of Algorithm~\ref{alg1} to 
a global minimizer of (P).

\begin{theorem} \label{conv_theorem}
 Let $\{\omega_s\}$ be the sequence of iterates generated by Algorithm~\ref{alg1}, in the general case
 when the support functions (\ref{eq:quad_model_multi}) are used on lines 3 and 11. \textbf{(i)} Every limit point of this 
 sequence is a global minimizer of the problem (P). \textbf{(ii)} Furthermore $\lim_{s \rightarrow \infty} u_s = \lim_{s \rightarrow \infty} l_s = \lambda^\ast$.
\end{theorem}
\begin{proof}
Since ${\cal B}$ is a bounded subset of ${\mathbb R}^d$, it follows that the sequence $\{\omega_s\}$ has at least one limit point $\omega^* \in {\cal B}$. 
By passing to a subsequence if necessary, we may assume that $\{\omega_s\}$ itself is a convergent sequence. Let $l^*$ 
denote the limit of the bounded nondecreasing sequence $\{l_s\}$. Since $l_s \leq \lambda^* \leq \lambda(\omega_s)$ for each 
$s \geq 0$, it suffices to show that $l^* = \lambda(\omega^*)$.

Suppose, for a contradiction, that there exists a real number $\delta > 0$ such that 
\begin{equation} \label{contra1}
 \lambda(\omega^*) \geq l^* + \delta.
\end{equation}
By the continuity of $\lambda$, there exists $s_1 \in \N$ such that 
\begin{equation} \label{def_s1}
 \lambda(\omega_s) \geq l^* + \frac{\delta}{2}, \quad \textrm{for all}~s \geq s_1.
\end{equation}
Since $\omega^*$ is the limit of the sequence $\{\omega_s\}$, there exists $s_2 \in \N$ such that
\begin{equation}\label{def_s2}
\|\omega_{s^\prime} - \omega_{s^{\prime \prime}}\|  <  \min\left\{\sqrt{\frac{\delta}{6 | \gamma |}}, \frac{\delta}{12 \mu} \right\}, \quad \textrm{for all}~s^\prime \geq s^{\prime \prime} \geq s_2,
\end{equation}
where we define $1/\mu := +\infty$ if $\mu = 0$, and $1/|\gamma | := +\infty$ if $\gamma = 0$. 
Let $s_* = \max\{s_1,s_2\}$. For each $s \geq s_*$, it follows from the definition of the functions 
$\overline{q}_s(\omega)$ that
\begin{eqnarray*}
 \overline{q}_s(\omega_{s+1}) & \geq & \overline{q}_{s_*} (\omega_{s+1}), \\
				      & \geq & q_{s_*}(\omega_{s+1}), \\
 		& = & \lambda(\omega_{s_*}) + \nabla \tilde{\lambda}_{j_*}(\omega_{s_*})^T (\omega_{s+1} - \omega_{s_*}) + \frac{\gamma}{2} \|\omega_{s+1} - \omega_{s_*}\|^2.
\end{eqnarray*}
where $j_* \in \{1,\dots,n \}$ is the index of the function that determines the
value of $q_{s_*}(w_{s+1})$ (see 
(\ref{eq:quad_model_multi})). Now, by applying the
Cauchy-Schwarz and triangle inequalities, and then using the inequalities
(\ref{def_s1}) and (\ref{def_s2}), we arrive at
\begin{eqnarray*}
  \overline{q}_s(\omega_{s+1})  & \geq & \lambda(\omega_{s_*}) -  \| \nabla \tilde{\lambda}_{j_*} (x^*)\| \|\omega_{s+1} - \omega_{s_*}\| - \frac{| \gamma |}{2} \|\omega_{s+1} - \omega_{s_*}\|^2, \\
				      & \geq & \left( l^* + \frac{\delta}{2} \right) -  
				      		    \left( \mu \cdot  \frac{\delta}{12\mu} \right) - 
						    \left( \frac{| \gamma |}{2} \cdot \frac{\delta}{6 | \gamma |} \right), \\
 				      & = & l^* + \frac{\delta}{3}.
\end{eqnarray*}
Using the definition $l_{s+1} = \overline{q}_s(\omega_{s+1})$, it follows that
\[
 l_{s+1} \geq l^* + \frac{\delta}{3}, \quad \textrm{for all}~s \geq s_*.
\]
Since $\delta > 0$, this contradicts our assumption that $l^*$ is the limit of the non-decreasing 
sequence $\{l_s\}$. Therefore, we have $\lambda(\omega^*) < l^* + \delta$ for all $\delta > 0$, or equivalently 
$\lambda(\omega^*) \leq l^*$.  Since $l_s \leq \lambda(\omega)$ for all $s \in \N$ and $\omega \in {\mathcal B}$, 
it follows that $l^* \leq \lambda(\omega^\ast)$, which establishes that $\lambda(\omega^*) = l^* \leq \lambda(\omega)$ for all 
$\omega \in {\cal B}$. Therefore, $\omega^*$ is a global minimizer of (P).  Moreover, $\lim_{s\rightarrow \infty} l_s = \lambda^*$.
The sequence $\{ u_s \}$ must also converge to $\lambda^*$, which can be deduced by observing 
$\lambda^* \leq u_s \leq \lambda(\omega_s)$ for each $s \in \N$ and taking the limit as $s \rightarrow \infty$. 
The proof of assertion \textbf{(i)} is completed by repeating the same argument 
for any other limit point of the sequence $\{\omega_s \}$.   

Assertion \textbf{(ii)} can be concluded by noting that the monotone, 
bounded sequences $\{ l_s \}$ and $\{ u_s \}$ must converge, and they
have subsequences converging to $\lambda^*$.
\end{proof}

\section{Numerical Experiments}\label{sec:eig_opt}

We compare Algorithm \ref{alg1} with the following algorithms: 
	\begin{enumerate}
		\item[\bf{(1)}] a brute force approach;
		\item[\bf{(2)}] the Piyavskii-Shubert algorithm
\cite{Piyavskii1972, Shubert1972} (only one-dimensional case);
		\item[\bf{(3)}] DIRECT method \cite{Jones1993};
		\item[\bf{(4)}] the specialized level-set based algorithms
whenever possible.
	\end{enumerate}
The brute force approach \textbf{(1)} splits the box  ${\mathcal B}$ into sub-boxes 
of equal side-lengths  and the eigenvalue function is computed at the corners of the sub-boxes.
 Algorithms \textbf{(2)} and \textbf{(3)} are global optimization techniques 
based on Lipschitz continuity of the function. The latter method \textbf{(3)}
benefits from several Lipschitz constant estimates simultaneously, 
while the former one \textbf{(2)} utilizes a global Lipschitz constant.
Algorithms that fall into the category \textbf{(4)} are  
level-set based approaches. They typically converge fast, but each iteration is
costly. Each of the algorithms in \textbf{(4)}
is devised for a particular eigenvalue optimization problem.

\vskip 1ex

\noindent
\textbf{Example 1 (Numerical Radius):}
This one-dimensional example concerns the calculation of the numerical radius (defined and 
motivated in Section \ref{sec:app_dyn_sys}) of an $n\times n$ matrix $A$. The matrix-valued function 
involved is ${\mathcal A}(\theta) = -(Ae^{i\theta} + A^\ast e^{-i\theta})/2$ and
$\lambda_n(\theta) := \lambda_n({\mathcal A}(\theta))$ is sought to be minimized over all $\theta \in [0,2\pi]$.
The numerical radius of $A$ corresponds to the negative of this globally minimal value of $\lambda_n(\theta)$.
Here, we assume that the generic analyticity holds, that is the eigenvalue $\lambda_n(\theta)$
is simple for all $\theta$. The derivative 
\begin{equation}\label{eq:num_rad_1der}
	\frac{d \lambda_n(\theta)} {d\theta}
				=
		\Im
		\left(
		v_n^{\ast}(\theta)
			Ae^{i\theta}
		v_n(\theta)
		\right)
\end{equation}
can be deduced from (\ref{eq:eigval_1der}).
We numerically observe that typically $\lambda_n''(\theta) \geq -2\| A\|$ holds
for all $\theta$, and set $\gamma = - 2\| A \|$.  We specifically focus on matrices
\[
	A_n = P_n	- 	(n/20) \cdot i R_n
\]
of various sizes, where $P_n$ is an $n\times n$ matrix obtained from a finite
difference discretization of the Poisson operator, and $R_n$ is a random $n\times n$ matrix
with entries selected from a normal distribution with zero mean and unit variance.
This is a carefully chosen challenging example, as $\lambda_n(\theta)$ has 
many local minima (see Figure \ref{fig:f_numrad_Poisson}). 

In Table \ref{tab:num_rad_fevals}, 
the function evaluations of algorithms \textbf{(1)-(3)} are given along with the function evaluations 
of Algorithm \ref{alg1} for the matrix $A_n$ with $n = 400$ and with respect to absolute accuracy. 
Algorithm \ref{alg1} - called eigopt in the table - converges 
linearly; this is evident from about fixed number of function evaluations
required for every two-decimal-digit accuracy.
All other algorithms, including DIRECT method, converge sublinearly. The computed global minimizer
is marked (with an asterisk) on a plot of $\lambda_n(\theta)$ with 
respect to $\theta \in [0, 2\pi]$ in Figure \ref{fig:f_numrad_Poisson} on the left. 
In the same figure on the right, the first five quadratic support functions formed are shown for the same example. 
The level-set based algorithm \textbf{(4)} for the numerical radius \cite{Mengi2005}  requires the solutions 
of eigenvalue problems twice the size of ${\mathcal A}(\theta)$. It is not included in 
Table \ref{tab:num_rad_fevals}, because these larger 
eigenvalue problems dominate the computation time rather than the calculation of $\lambda_n(\theta)$. 
Instead we compare the CPU times (in seconds) of Algorithm \ref{alg1} and this
specialized algorithm on Poisson 
matrices $A_n$ of various sizes $n$ in Table \ref{tab:num_rad_cpu}. 
Algorithm \ref{alg1} is run to retrieve  the results with at least 10-decimal-digit accuracy. 
The table displays the superiority of the running times 
of Algorithm \ref{alg1} as compared to those of the level-set approach.

\begin{table}[h]
\begin{center}
\begin{tabular}{|c||llllll|}
\hline
$\epsilon$		&	$10^{-2}$	&	$10^{-4}$	&	$10^{-6}$	&	$10^{-8}$		&	$10^{-10}$	& 	$10^{-12}$	\\	
\hline
\hline
eigopt			&	46 		&	59 		&	69 		&	79		&	89		&	98			\\
brute force		&	881		&	8812		&	88125	&	881249	&	8815191	&	86070462 	\\	
Piyavskii-Shubert	&	1907		&	18817	&	--		&	--		&	--		&	--			\\
DIRECT			&	25		&	51		&	61		&	105		&	245		&	597			\\
\hline
\end{tabular}
\caption{ Number of function evaluations by various algorithms to compute the numerical
radius of the $400\times 400$ Poisson example (Example 1) with respect to absolute accuracy $\epsilon$.
Note that eigopt makes use of the derivatives in addition to the eigenvalues evaluated. However,
their derivatives come essentially at no cost once the eigenvalues are evaluated.
} \label{tab:num_rad_fevals}
\end{center}
\end{table}

\begin{figure}[h]
	\begin{tabular}{ll}
		\includegraphics[width=.48\textwidth, height=.40\textheight]{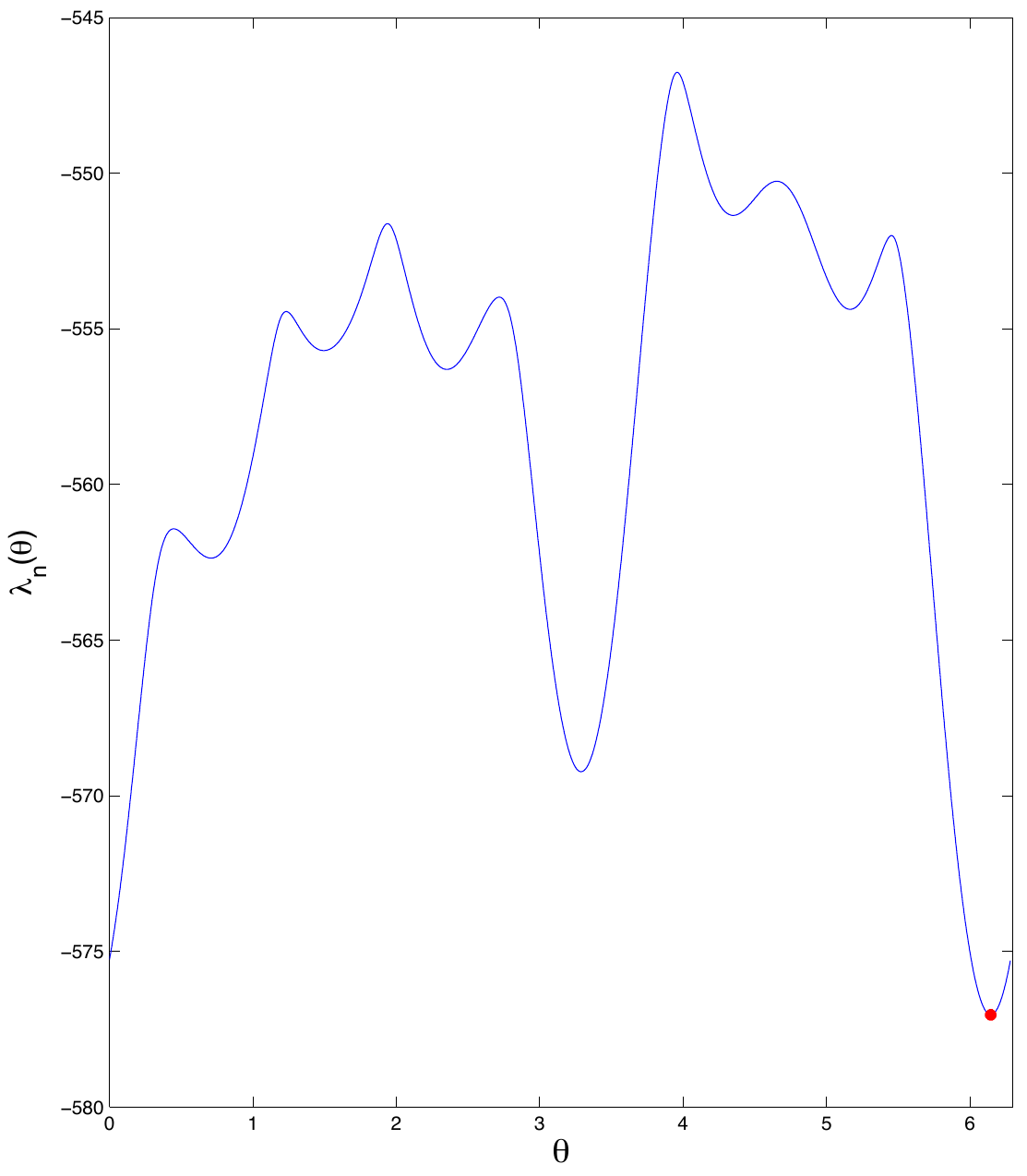} &
		\includegraphics[width=.51\textwidth, height=.40\textheight]{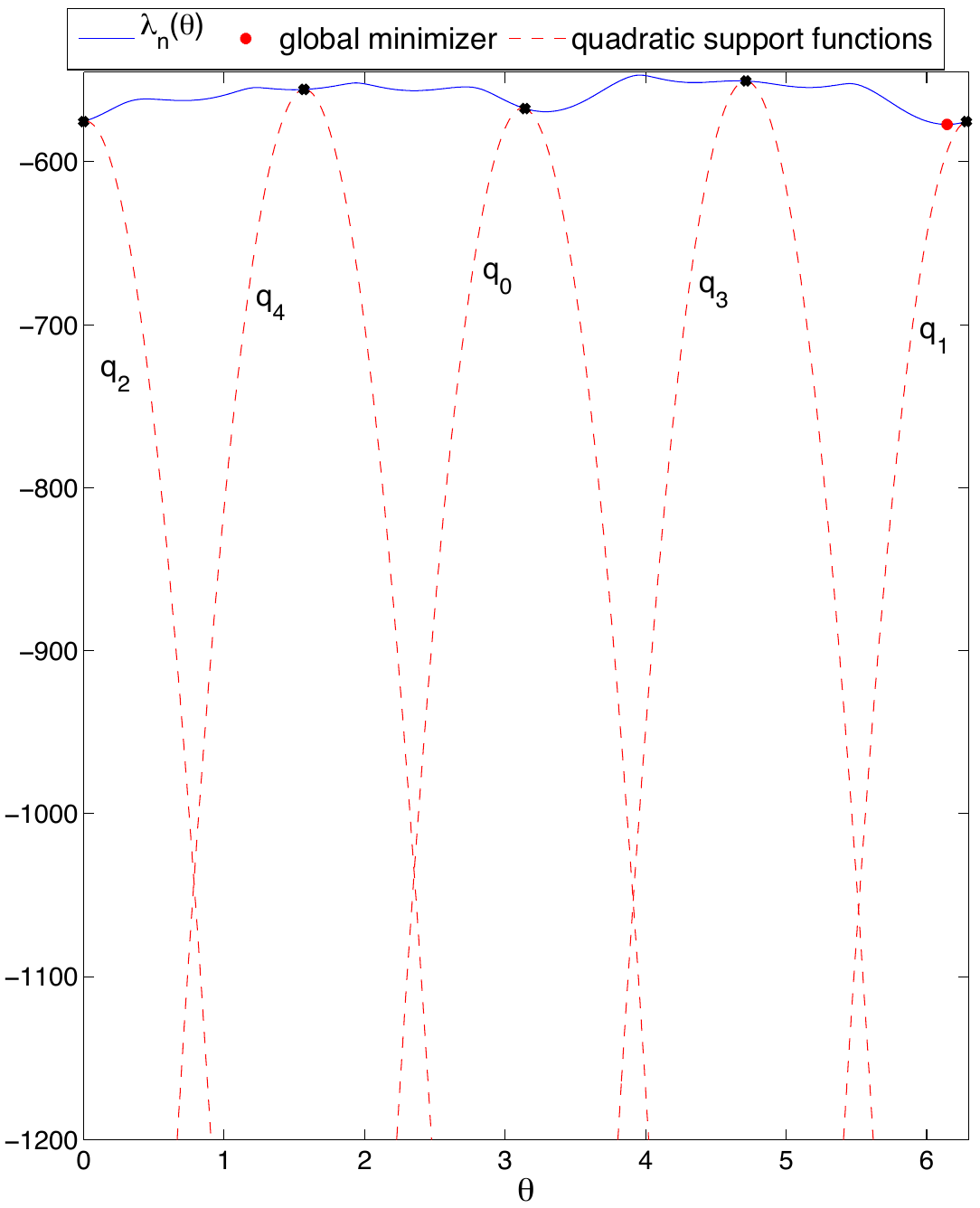} \\
	\end{tabular}
	    \caption{  On the left, the plot of  $\lambda_n({\mathcal A}(\theta))$ for the 
	    $400\times 400$ Poisson example (Example 1) with respect to $\theta$; 
	    On the right, the plot of  $\lambda_n({\mathcal A}(\theta))$ together with the
	    first five quadratic support functions $q_j, \; j = 0, \dots, 4$ formed.
	    In both figures, the asterisk represents the global minimizer, the squares
	    on the right mark the iterates of the algorithm around which the support functions
	    are constructed. }\label{fig:f_numrad_Poisson}
\end{figure}

 \begin{table}[h]
 \begin{center}
	\begin{tabular}{|c||ccccc|}
	\hline
		  $n$			& 400	 & 900 	&	 1600 	& 2500 	    & 3600	\\
	\hline
	\hline
		eigopt		& 14		 &  103	&	 328 	 	  &	1079       &	2788		\\
		level-set		& 17		 &  181	&	 1477	  &	--           &	--	\\
	\hline
	\end{tabular}
 \end{center}
 \caption{ CPU times (in seconds) required by Algorithm \ref{alg1} (eigopt) and
the specialized
 level-set approach in \cite{Mengi2005} to compute the numerical radius of Poisson
 matrices of varying size $n$. }\label{tab:num_rad_cpu}
 \end{table}

\vskip 1ex

\noindent
\textbf{Example 2 (Distance to Uncontrollability):}
Next, we consider the distance to uncontrollability defined and motivated in
Section \ref{sec:app_dyn_sys}.
Here, for a given linear system 
$ \;\;
		x'(t) = Ax(t) + Bu(t),
  \;\;
$ 
where $A \in {\mathbb C}^{n\times n}, B\in {\mathbb C}^{n\times m}$ are such that $n\geq m$,
the smallest singular value of 
			$
				{\mathcal A}(z)
					=
				\left[
				\begin{array}{cc}
					A - z I	&	B
				\end{array}
			\right]$
is sought to be minimized over $z \in {\mathbb C}$.  We again assume generic simplicity, that is the 
multiplicity of $\sigma_n(z) := \sigma_n \left( {\mathcal A}(z) \right)$ is one for all $z$. This property holds on a dense subset
of all pairs $(A, B)$.  In this case, expressions for the gradient is given by
(\ref{eq:first_der_sval}). In particular, 
denoting a consistent pair of unit left and right singular vectors associated with $\sigma_n(z)$ 
by $u_n(z) \in {\mathbb C}^n, 
	v_n(z)
		=
		\left[
			\begin{array}{c}
				\tilde{v}_n(z) \\
				\hat{v}_n(z) \\
			\end{array}
		\right] \in {\mathbb C}^{n+m}$ 
where $\tilde{v}_n(z) \in {\mathbb C}^n, \hat{v}_n(z) \in {\mathbb C}^m$, we have
 \[
	\nabla \sigma_n(z)
			=
	\left(
		\frac{\partial \sigma_n (z)}{\partial \Re z}, \frac{\partial \sigma_n (z)}{\partial \Im z}
	\right)
			=
	\left(
	\;
	-\Re
	\left(
		u_n^{\ast}(z)\tilde{v}_n(z)
	\right), \;
	\Im
	\left(
		u_n^{\ast}(z)\tilde{v}_n(z)
	\right)
	\;
	\right).
\]
Furthermore, $\gamma = -4$ appears to be a good lower bound for $\lambda_{\min} \left( \nabla^2 \sigma_n (z) \right)$ numerically.
We perform tests on linear systems $(A,B)$ arising from a discretization of the heat equation, taken from 
the SLICOT library, see \cite[Example 3.2]{Kressner1998}. The matrix $A$ is real, symmetric, tridiagonal, and $n\times n$, 
whereas $B$ is real and $n\times 1$. 

In Table $\ref{tab:minimize_du}$, the number of function evaluations required by
global optimization algorithms for Lipschitz continuous functions, and
Algorithm \ref{alg1} are presented when the order of the
system satisfies $n = 30$. The Piyavskii-Shubert algorithm is omitted, because
it would be based on locating global
minimizers of piecewise cones, and it is not immediate how one would locate these
minimizers. Once again, the number of function
evaluations for Algorithm \ref{alg1} seems to suggest linear convergence.

The progress of the algorithm can be traced 
from the graph associated with it. Recall that the box is split into subregions
(indeed polytopes). Inside each subregion, 
one of the quadratic support functions dominates the others. For the heat
equation example of order $n = 30$, these graphs 
are provided after 60, 360, and 580 iterations in Figure
\ref{fig:graph_progress} in the top three plots. The bottom plot 
in Figure \ref{fig:graph_progress} is an illustration of the level sets of $\sigma_n(z)$ in the complex 
plane along with the computed global minimizer (accurate up to 13 decimal digits after 589 function evaluations) marked 
with an asterisk. The box is split into subregions more or less uniformly
initially, e.g., after 60 iterations. However, later 
iterations form finer subregions around minimizers of $\sigma_n(z)$. 

We also compare Algorithm \ref{alg1}
with the level-set approach in \cite{Gu2006} on the heat equation examples of varying order. The level set approaches 
become prohibitively expensive as the number of optimization parameters increases. Here, with the minimization over two  
parameters, it requires the solutions of eigenvalue problems of size $n^2$ for a system of order $n$. For small systems, 
the level-set approach works very well, however, even for medium-scale systems
it becomes computationally infeasible. This is illustrated 
in Table \ref{tab:du_rad_cpu}. On the other hand, Algorithm \ref{alg1} is capable
of solving even a problem of order 1000 
in a reasonable amount of time.

\begin{table}
\begin{center}
\begin{tabular}{|c||llllll|}
\hline
$\epsilon$		&	$10^{-2}$	&	$10^{-4}$	&	$10^{-6}$	&	$10^{-8}$		&	$10^{-10}$	& 	$10^{-12}$		\\	
\hline
\hline
eigopt			&	520		&	531				&	543						&	556					&	572					& 	585			  \\
brute force		&	$4.2\times 10^6$	&	$4.2\times 10^{10}$	&	$4.2\times 10^{14}	$		&	$4.2\times 10^{18}$	&	$4.2\times 10^{22}$	&  $4.2\times 10^{26}$	\\	
DIRECT			&	37781		&	37867		&	37867		&	38233		&	38441		&	38805			\\
\hline
\end{tabular}
\caption{ Number of function evaluations by various algorithms to compute the distance to uncontrollability
for the linear system $(A,B)$ arising from the heat equation \cite[Example
3.2]{Kressner1998}, where 
$A\in {\mathbb R}^{30\times 30}$, $B\in {\mathbb R}^{30\times 1}$ with respect to absolute 
accuracy $\epsilon$.} \label{tab:minimize_du}
\end{center}
\end{table}

\begin{figure}
   	\begin{tabular}{c}
		\includegraphics[width=\textwidth,]{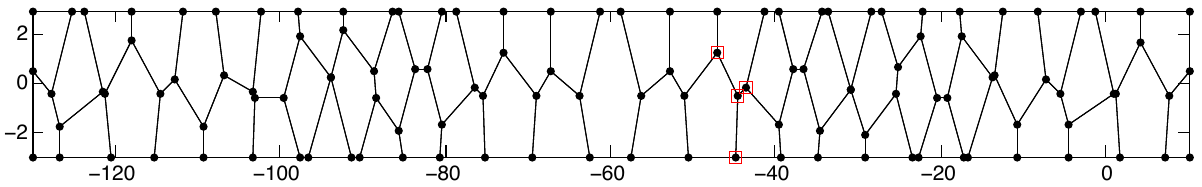} \\
		\includegraphics[width=\textwidth,height=0.12\textheight]{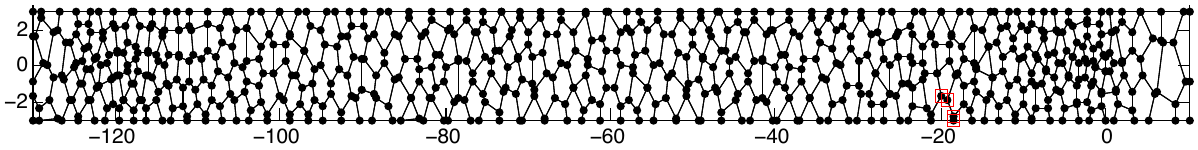} \\
		\includegraphics[width=\textwidth,height=0.12\textheight]{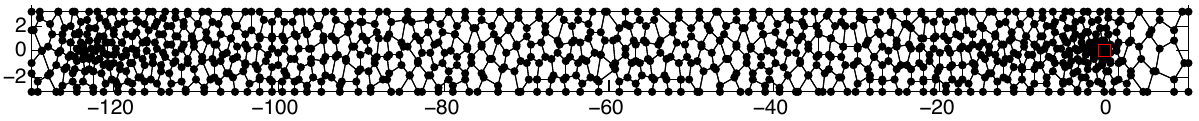} \\
		\includegraphics[width=\textwidth,]{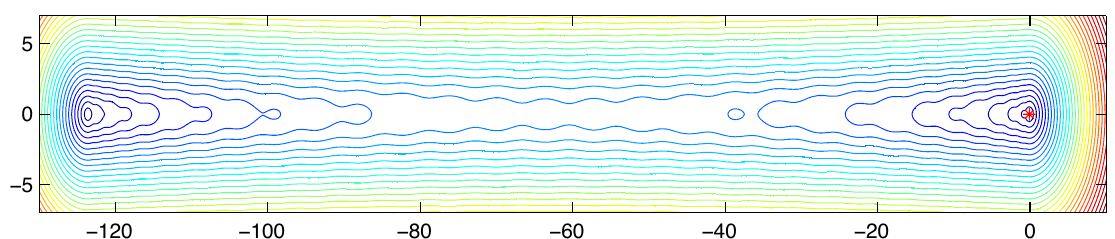} \\
	\end{tabular}
	    \caption{ The top three plots display the graph associated with Algorithm \ref{alg1} after
	    60, 360, and 580 iterations to compute the distance to uncontrollability for the heat equation
	    example \cite[Example 3.2]{Kressner1998}  of order $n=30$ (see Example 2). The red squares 
	    correspond to dead vertices. The bottom plot displays the level sets of $\; \sigma_n(z)$
	    for the same heat equation example together with the computed global minimizer marked
	    with an asterisk. }\label{fig:graph_progress}
\end{figure}

 \begin{table}
 \begin{center}
	\begin{tabular}{|c||cccccc|}
	\hline
		  $n$			& 30	 	& 		60 		&	 100 		& 	200 	    &   400	&	1000	 \\
	\hline
	\hline
		eigopt		& 45		 &	 	 46	&	 56 	 	  &	85       &	119	&	384	  \\
		level-set		& 20		 &		 393	&	 --	  	&	--          &	-- 	& 	--	\\
	\hline
	\end{tabular}
 \end{center}
 \caption{ CPU times (in seconds) required by Algorithm \ref{alg1} (eigopt) and
the level-set approach in \cite{Gu2006} to compute the distance to uncontrollability 
for the heat equation examples (Example 2) with respect to the order of the system $n$. }\label{tab:du_rad_cpu}
 \end{table}

\vskip 1ex

\noindent
\textbf{Example 3 (Minimizing the Largest Eigenvalue, Convex):}
This example is taken from \cite{Fan1995}, and concerns the minimization of the largest eigenvalue of 
an affine matrix function of the form
\[
	{\mathcal A}(\omega)	= 	A_0	+ \sum_{j=1}^5 \omega_j A_j,
\]
depending on five parameters, where $A_j \in {\mathbb R}^{5\times 5}, \;
j=0,\dots,5$, are real, symmetric and as in \cite{Fan1995},
where the minimal value of $\lambda_1(\omega) := \lambda_1({\mathcal A}(\omega))$ is cited as $0.708882597$.
Whenever ${\mathcal A}(\omega)$ is affine, the largest eigenvalue 
$\lambda_1(\omega)$ is convex by Theorem \ref{thm:extreme_low_bound} (see also
\cite[Theorem A.1]{Fletcher1985}), so we set $\gamma = 0$. Thus, this eigenvalue optimization problem can be 
posed as a semi-definite program (SDP), and solved by means of interior-point
methods \cite{Nesterov1994}. 
The purpose here is to compare the performances of Algorithm \ref{alg1}, and DIRECT method on a multi-dimensional 
example for which the solution is known (even though our algorithm is
really devised for non-convex problems,
and in all likelihood an interior-point method would outperform it in the convex
case). 

This comparison is provided in
Table \ref{tab:minimize_largest_convex}, where we again observe linear
convergence of Algorithm \ref{alg1}. DIRECT method is not suitable even for a few decimal-digit
precision. Indeed, even after
500,000 function evaluations and 3732 seconds of CPU time, it cannot achieve six
decimal-digit accuracy. 
In contrast to the one-dimensional and two-dimensional examples, keeping the
graph structure properly, in particular forming the adjacencies
between the new vertices, take almost all of the computational time rather than the function evaluations.  
Even if the matrices $A_j$ were much bigger than $5\times 5$, the computation times would not be affected 
significantly. Therefore, in the table, the CPU times (in seconds) are also
provided in parenthesis. 

The later iterations are more expensive, since the number of new vertices created increases at the
later iterations. Up to four dimensions, the increase in the number of vertices
at every iteration seems 
more or less fixed. In contrast, this does not hold when the dimension is five
or more.
(These observations are solely based on numerical experiments; we do not have a
clear understanding 
of this phenomenon at the moment.) 
This is illustrated for the affine example with five
parameters in Figure 
\ref{fig:vertices}. In the figure on the left, the solid and dashed lines
represent the number of dead vertices
and the number of newly added vertices, respectively, at iterations $1, 2,
\dots, 107$.
The graph indicates an increase in
both the number of dead vertices and the number of new vertices with respect to
the iteration. However, the
increase in the number of new vertices is larger. On the right, the total number
of vertices is displayed
with respect to the iteration number. The graph reveals that the number of
vertices seems to increase superlinearly.

\begin{table}
\begin{center}
\begin{tabular}{|c||llllll|}
\hline
$\epsilon$		&	$10^{-2}$	&	$10^{-4}$	&	$10^{-6}$	&	$10^{-8}$		&	$10^{-10}$	& 	$10^{-12}$		\\	
\hline
\hline
eigopt			&	17 (34)	&	34 (203)		&	50 (501)	&	67	(1021)	&	87 (2040)		& 	108 (3725)	\\
DIRECT			&	459 		&	45937	&	--		&	--			&	--			&	--			\\
\hline
\end{tabular}
\caption{ Number of function evaluations by Algorithm \ref{alg1} (eigopt) and DIRECT
method for minimizing the
largest eigenvalue of the affine matrix function (Example 3) with respect to absolute
accuracy $\epsilon$. The CPU times
for eigopt (in seconds) are also provided in parentheses. } \label{tab:minimize_largest_convex}
\end{center}
\end{table}

\begin{figure}
   	\begin{tabular}{ll}
		\includegraphics[width=.50\textwidth,]{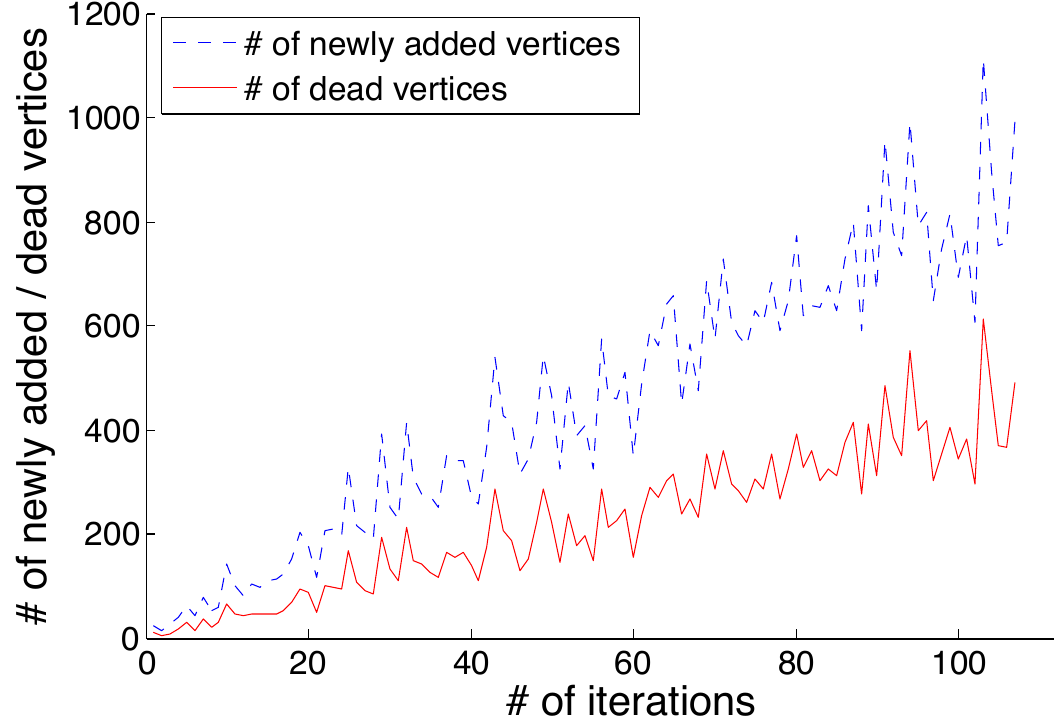} &
		\includegraphics[width=.47\textwidth,]{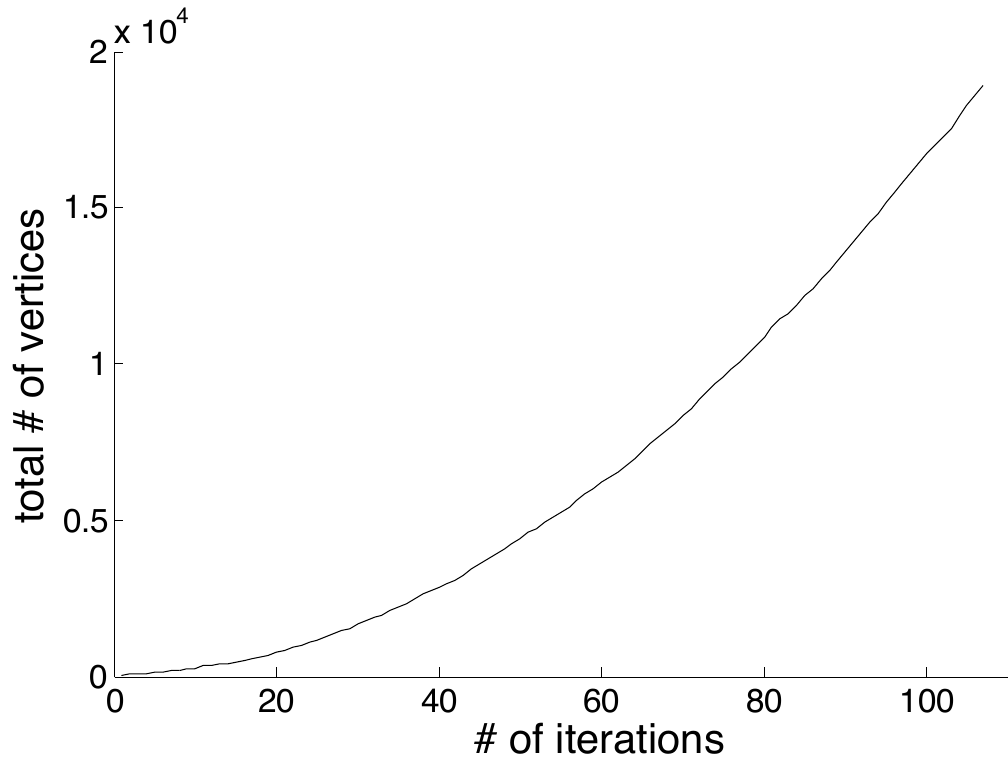} \\
	\end{tabular}
	    \caption{  On the left, the number of dead vertices (solid line) and
the number of
	    newly added vertices (dashed line) with respect to the iteration number are displayed,
	    whereas on the right, the total number of vertices with respect to
the iteration number
	    is displayed for Example 3. }\label{fig:vertices}
\end{figure}

\vskip 1ex

\noindent
\textbf{Example 4 (Minimizing the Largest Eigenvalue, Non-convex):}
This is a non-convex example depending on four parameters,  involving the minimization of the 
largest eigenvalue of the matrix function
\begin{equation}\label{eq:quad_mat_func}
	{\mathcal A}(\omega) 
			=
		A_0
			+
		\sum_{j=1}^4 \omega_j A_j
			+
		\frac{1}{2} \sum_{j=1}^4 \sum_{k=1}^4 \omega_j \omega_k A_{jk}, 
\end{equation}
where $A_0, A_1, \dots, A_4$ are as in the previous affine example taken from \cite{Fan1995}, while $A_{jk}$ are randomly
chosen $5\times 5$ symmetric and real matrices such that $A_{jk} = A_{kj}$. By Corollary \ref{cor:extreme_quad_low_bound},
a lower bound for $\lambda_{\min} \left( \nabla^2 \lambda_1 \left( {\mathcal A}(\omega)  \right)  \right)$
is given by $\gamma = \lambda_{\min} \left( \nabla^2 {\mathcal A} \right)$ where $\nabla^2 {\mathcal A} \in {\mathbb R}^{20\times 20}$,
and $A_{jk}$ constitutes the $5\times 5$ submatrix of $\nabla^2 {\mathcal A}$ at rows $5(j-1)+1:5j$ and columns $5(k-1)+1:5k$.
Table \ref{tab:minimize_largest_nonconvex} displays a comparison of the number of function
evaluations for Algorithm \ref{alg1} and DIRECT method with respect to
accuracy. As usual, Algorithm \ref{alg1} seems to exhibit linear convergence.

\begin{table}
\begin{center}
\begin{tabular}{|c||llllll|}
\hline
$\epsilon$		&	$10^{-2}$	&	$10^{-4}$	&	$10^{-6}$	&	$10^{-8}$		&	$10^{-10}$	& 	$10^{-12}$		\\	
\hline
\hline
eigopt			&	52 		&	141 		&	225 		&	309			&	391			& 	472 			\\		%(662)
DIRECT			&	45		&	161		&	369		&	939			&	1367			&	2049			\\
\hline
\end{tabular}
\caption{ Number of function evaluations by Algorithm \ref{alg1} (eigopt) and DIRECT method for minimizing the
largest eigenvalue of the matrix function (\ref{eq:quad_mat_func}) with respect to absolute accuracy $\epsilon$. }
\label{tab:minimize_largest_nonconvex}
\end{center}
\end{table}

\section{Software}
A MATLAB implementation of Algorithm \ref{alg1} is available on the web. The implementation makes use of heaps, adjacency lists, 
stacks, and updates the underlying graph structure efficiently. The user is expected to write down a MATLAB routine 
that calculates the eigenvalue function and its gradient at a given point. The name of this routine and $\gamma$,
a global lower bound on the minimum eigenvalues of the Hessians of the eigenvalue functions must be
supplied by the user. We refer to the web page associated with this
implementation\footnote{http://home.ku.edu.tr/$\sim$emengi/software/eigopt.html}, where a user guide is also provided.

\section{Conclusion}\label{sec:conc}
The analytical properties of eigenvalues of matrix-valued functions facilitate
the use of so-called quadratic support functions that globally underestimate the
eigenvalue curves for their optimization. This observation motivates the idea of
adapting support function based global
optimization approaches, especially the approach due to Breiman and Cutler \cite{Breiman1993},
for non-convex eigenvalue optimization. In this paper, we illustrated how such
global optimization
approaches based on the derivative information could be realized in the context
of non-convex
eigenvalue optimization. We derived the necessary quadratic support functions, elaborated
on deducing analytical global lower bounds $\gamma$ for the second derivatives of the extreme
eigenvalue functions, which are essential for the algorithm, and provided a global convergence
proof. The algorithm is especially applicable for the optimization of extreme
eigenvalues, for instance, for the minimization of the largest eigenvalue, and
those eigenvalue
functions that exhibit generic simplicity.

\vskip 2ex

\noindent
\textbf{Acknowledgements} We thank two anonymous referees for their invaluable comments,
which improved this paper considerably. We are also grateful to Michael Karow and Melina
Freitag for helpful discussions and feedback.
% concerning the derivatives of eigenvalues of analytic 
% Hermitian matrix-valued functions, and
% Melina Freitag for helpful comments.

\bibliography{eig_optimization}
\end{document}